\documentclass{amsart}
\usepackage{hyperref}
\usepackage{mathtools}
\usepackage{enumitem}
\usepackage{graphicx}
\usepackage{tikz}

\newtheorem{theorem}{Theorem}[section]
\newtheorem{proposition}[theorem]{Proposition}
\newtheorem{lemma}[theorem]{Lemma}
\newtheorem{corollary}[theorem]{Corollary}
\theoremstyle{remark}
\newtheorem{remark}[theorem]{Remark}
\newtheorem{example}[theorem]{Example}
\newtheorem{question}[theorem]{Question}

\newcommand{\Mod}{\mathrm{Mod}}
\newcommand{\T}{\mathcal{T}}

\newcommand{\A}{\mathcal{A}}
\newcommand{\B}{\mathcal{B}}
\newcommand{\F}{\mathcal{F}}
\newcommand{\M}{\mathcal{M}}
\newcommand{\N}{\mathcal{N}}
\newcommand{\ML}{\mathcal{ML}}

\newcommand{\diam}{\mathrm{diam}}
\newcommand{\X}{\mathcal{X}}
\newcommand{\Y}{\mathcal{Y}}
\newcommand{\Z}{\mathcal{Z}}

\newcommand{\floor}[1]{\left\lfloor #1 \right\rfloor}
\definecolor{OliveGreen}{rgb}{0,0.6,0}

\newtheoremstyle{TheoremNum}
{\topsep}{\topsep}              
{\itshape}                      
{}                              
{\bfseries}                     
{.}                             
{ }                             
{\thmname{#1}\thmnote{ \bfseries #3}}
\theoremstyle{TheoremNum}
\newtheorem{thmn}{Theorem}

\begin{document}

\title{Growth Rate Of Dehn Twist Lattice Points In Teichm\"{u}ller Space}
\author{Jiawei Han}
\address{Department of Mathematics, Vanderbilt University, Nashville, Tennessee 37235}
\email{jiawei.han@vanderbilt.edu}

\begin{abstract}
Athreya, Bufetov, Eskin and Mirzakhani \cite{athreya2012lattice} have shown the number of mapping class group lattice points intersecting a closed ball of radius $R$ in Teichm\"{u}ller space is asymptotic to $e^{hR}$, where $h$ is the dimension of the Teichm\"{u}ller space. In contrast we show the number of Dehn twist lattice points intersecting a closed ball of radius $R$ is coarsely asymptotic to $e^{\frac{h}{2}R}$. Moreover, we show the number multi-twist lattice points intersecting a closed ball of radius $R$ grows coarsely at least at the rate of $R \cdot e^{\frac{h}{2}R}$.
\end{abstract}

\maketitle

\section{Introduction}
\subsection{Motivation}
Let $M$ be a compact, negatively curved Riemannnian manifold and denote $\tilde{M}$ its universal cover. Then its fundamental group $\pi_1(M)$ acts on $\tilde{M}$ by isometries. Given any $x \in \tilde{M}, R > 0$, let $B_R(x)$ denote the ball of radius $R$ in $\tilde{M}$ centered at $x$. It's a classical result from G.A. Margulis that
\begin{theorem}[Margulis \cite{margulis2004some}]
	There is a function $c \colon M \times M \to \mathbb{R}^+$ so that for every $x, y \in \tilde{M}$,
	\begin{align*}
	|\pi_1(M) \cdot y \cap B_R(x)| \sim c(p(x),p(y))e^{hR}
	\end{align*}
	where $h$ equals to the dimension of the manifold, which is the topological entropy of the geodesic flow on the unit tangent bundle of $M$. 
\end{theorem}
Here and throughout, the notation $f(R) \sim g(R)$ means $\lim_{R \to \infty} \frac{f(R)}{g(R)} = 1$.

Inspired by this result, Athreya, Bufetov, Eskin and Mirzakhani studied lattice point asymptotics in Teichm\"{u}ller space. Let $S_g, g \ge 2$ denote the closed surface of genus $g$, and let $\Mod_g, \T_{g}$ denote the corresponding mapping class group and Teichm\"{u}ller space respectively, then $\Mod_g$ acts on $\T_{g}$ by isometries. They showed

\begin{theorem}[Athreya, Bufetov, Eskin and Mirzakhani \cite{athreya2012lattice}] \label{ABEM}
	For any $\X, \Y \in \mathcal T_g$, we have
	\begin{align*}
	| \Mod_g \cdot \Y \cap B_R(\X)| \sim \Lambda(\X)\Lambda(\Y) e^{hR}
	\end{align*}
	where $h= 6g-6$ is the dimension of $\T_{g}$ and $\Lambda\colon \T_{g} \to \mathbb{R}^{+}$ is a bounded function called Hubbard-Masur function.
\end{theorem}
We note that the function $\Lambda$ was later shown to be a constant by Mirzakhani, see \cite{10.1007/s11511-015-0129-6}.

The Nielsen-Thurston Classification \cite{thurston1988} says every element in $\Mod_g$ is one of the three types: periodic, reducible, or pseudo-Anosov. Letting $PA \subset \Mod_g$ denote the subset of pseudo-Anosov elements, Maher showed that, in terms of lattice points counting, pseudo-Anosov elements are generic in the following sense. 
\begin{theorem}[Maher \cite{maher2010asymptotics}] \label{Maher}
	For any $\X, \Y \in \mathcal T_g$, we have
	\begin{align*}
	\frac{|PA \cdot \Y \cap B_R(\X)|}{| \Mod_g \cdot \Y \cap B_R(\X)|} \sim 1.
	\end{align*}
\end{theorem}
Note the above Theorems \ref{ABEM} and \ref{Maher} also hold for punctured surfaces $S_{g,n}$ satisfying $3g+n \ge 5$.

Thus, it is natural to consider the asymptotic growth rate of reducible and periodic elements. A typical reducible element can be decomposed as a product of Dehn twists about disjoint simple closed curves and a partial pseudo-Anosov element on subsurfaces \cite{birman1983}. Dehn twists are also in a sense the most fundamental elements of mapping class groups, being both relatively elementary, yet sufficient to generate the mapping class group \cite{10.1007/BFb0063188}. This motivates us to understand the asymptotic growth behavior of Dehn twists.

\subsection{Statement of Main Results}
Throughout this paper we let $S_{g,n}$ denote a closed surface of genus $g$ with $n$ punctures such that $3g-3+n > 0$, and we let $\Mod_{g,n}, \T_{g,n}$ and $\M_{g,n}$ denote the corresponding mapping class group, Teichm\"{u}ller space and moduli space respectively. We use $h=6g+2n-6$ to denote the dimension of $\T_{g,n}$. For any $\epsilon > 0$, we denote $\T_{g,n}^\epsilon$ the $\epsilon$-thick part of $\T_{g,n}$. By saying $\alpha$ is a simple closed curve on $S_{g,n}$, we mean it's a non-trivial isotopy class of essential simple closed curves on $S_{g,n}$.

A multicurve $\alpha$ is a formal sum $\alpha = \sum_{i=1}^k a_i \alpha_i$ where $1 \le k \le \frac{h}{2}, a_i \in \mathbb{Z} \setminus \{0\}$, and $\alpha_i$ are pairwise disjoint simple closed curves on $S_{g,n}$. By this definition, simple closed curves are also multicurves. Let $\ML(\mathbb{Z})$ denote the set of multicurves on $S_{g,n}$ and let $\mathcal S \subset \ML(\mathbb{Z})$ denote the set of all simple closed curves.  A multicurve $\alpha = \sum_{i=1}^k a_i \alpha_i$ is said to be positive if all coefficients are positive and is said to be negative if all coefficients are negative. Otherwise, we say $\alpha$ is of mixed sign. Two multicurves are of the same topological type if up to isotopy, there is an orientation-preserving homeomorphism taking one multicurve to another. For any $\gamma \in \ML(\mathbb{Z})$, we denote the multicurves of topological type $\gamma$ by $\ML(\gamma)$. In particular if $\gamma$ is a simple closed curve, we denote $\ML(\gamma)$ as $\mathcal S(\gamma)$. Since there are only finitely many topological types of simple closed curves on $S_{g,n}$, $\mathcal S$ is a finite union of sets of the form $\mathcal S(\gamma)$. Meanwhile, there are infinitely many topological types of multicurves, as can be seen by looking at the coefficients. 

For reasons we will see, let's denote $\ML^*(\mathbb{Z})$ the set of multicurves $\alpha = \sum_{i=1}^k a_i \alpha_i$ satisfying one of the following conditions.
\begin{enumerate}
	\item $\alpha$ is a weighted or unweighted simple closed curve, i.e., $k =1$.
	\item $\alpha$ is positive or negative, i.e., all coefficients have the same sign.
	\item $\alpha$ is of mixed sign where each $|a_i| \ge 2$.
\end{enumerate}
We write $\ML^*(\gamma)$ instead of $\ML(\gamma)$ when $\gamma \in \ML^*(\mathbb{Z})$.

For any simple closed curve $\alpha$ we let $T_\alpha$ denote the Dehn twist around $\alpha$. In general, for any multicurve $\alpha = \sum_{i=1}^k a_i \alpha_i$, we define $T_\alpha = \prod_{i=1}^{k} T_{\alpha_i}^{a_i}$ and we call this a multi-twist. By this definition, Dehn twists are also multi-twists, and let's call them as twists in general. In our Theorem \ref{Main Theorem} we will consider the following subsets of $\Mod_{g,n}$ consisting of twists:
\begin{enumerate}
	\item $D(\mathcal{ML}^*(\gamma)) = \{T_\alpha \mid \alpha \in \mathcal{ML}^*(\gamma)\} = \{fT_\gamma f^{-1} \mid f \in \Mod_{g,n}\}$, the set of twists about curves of topological type $\gamma$ or, equivalently, the conjugacy class of $T_\gamma$.
	
	\item $D(\mathcal S) = \{T_\alpha \mid \alpha \in \mathcal S\}$, the set of all Dehn twists without powers. $D(\mathcal S)$ is a finite union of sets of the form $D(\mathcal S(\gamma))$.
	
	\item $M(\mathcal S) = \{T_\alpha^k \mid \alpha \in \mathcal S, k \in \mathbb{Z} \}$, the set of all Dehn twists with any powers. Similarly, $M(\mathcal S)$ is a finite union of $M(\mathcal S(\gamma))$. Each $M(\mathcal S(\gamma))$ can be realized as an infinite union of conjugacy classes of $T_\gamma^k, k \in \mathbb{Z}$.
\end{enumerate}

We now introduce some notations. Let $A  > 0$. 
\begin{enumerate}
	\item We say $f(x) \stackrel{+A}{\asymp} g(x)$ if $g(x)-A\le f(x) \le g(x)+A$ for any $x$.
	\item We say $f(x) \stackrel{*A}{\asymp} g(x)$ if $\frac{1}{A} \cdot g(x)\le f(x) \le A \cdot g(x)$ for any $x$.
	\item We say $f(R) \stackrel{*A}\preceq g(R)$ if for any $\lambda > 1$, there exists a $M(\lambda)$ such that $\frac{1}{\lambda A} \cdot f(R) \le g(R)$ for any $R \ge M(\lambda)$.
	\item We say $f(R) \stackrel{*A}\sim g(R)$ if $f(R) \stackrel{*A}\preceq g(R)$ and $g(R) \stackrel{*A}\preceq f(R)$. 
\end{enumerate}
Moreover, we say $f,g$ are coarsely asymptotic if $f(R) \stackrel{*A}\sim g(R)$ for some coefficient $A$. Notice the notation $f(R) \sim g(R)$ is the same as $f(R) \stackrel{*1}\sim g(R)$, i.e. $f,g$ are asymptotic when they are coarsely asymptotic with coefficient 1.

Recall $\T_{g,n}^\epsilon$ denotes the $\epsilon$-thick part of $\T_{g,n}$ and $h = 6g-6+2n$ denotes the dimension of $\T_{g,n}$. For any $\X, \Y \in \T_{g,n}$, we define $F(\X,\Y) = e^{\frac{h}{2} d_\T(\X,\Y)}$. For any multicurve $\alpha = \sum_{i=1}^k a_i \alpha_i$, we denote the sum of absolute coefficients as $c_\alpha = \sum_{i=1}^k|a_i|$ and we define $F_\alpha(\X,\Y) =(c_\alpha)^\frac{h}{2} e^{\frac{h}{2} d_\T(\X,\Y)}$. Our main theorem below gives coarse asymptotics for $D(\ML^*(\gamma))$.
\begin{theorem}\label{Main Theorem}
Given any $S_{g,n}$ and given any $\epsilon > 0$, there exists a $J > 0$ such that for any multicurve $\gamma \in \ML^*(\mathbb{Z})$ and for any $\X, \Y \in \T_{g,n}^\epsilon$, we have
\begin{align*}
|D\left(\ML^*(\gamma)\right) \cdot \Y \cap B_R(\X)| \stackrel{*JF_\gamma(\X,\Y)}{\sim} n_X(\gamma) \cdot e^{\frac{h}{2} R}
\end{align*}
where $n_X(\gamma)$ is the corresponding Mirzakhani constant, see section \ref{Counting Simple Closed Geodesics}.
\end{theorem}
\begin{corollary}\label{Main Corollary}
Given $S_{g,n}$ and given any $\epsilon > 0$, for any $\X, \Y \in \T_{g,n}^\epsilon$, we have
	\begin{align*}
		&|D\left(\mathcal S\right) \cdot \Y \cap B_R(\X)| \stackrel{*JF(\X,\Y)}{\sim} n_X(\mathcal S) \cdot e^{\frac{h}{2} R},  \text{\ if\ } h>0, \\
		&|M\left(\mathcal S\right) \cdot \Y \cap B_R(\X)| \stackrel{*8JF(\X,\Y)}{\sim} n_X(\mathcal S) \cdot e^{\frac{h}{2} R},  \text{\ if\ } \frac{h}{2}>1
	\end{align*}
	where $n_X(\mathcal{S})$ is the corresponding Mirzakhani constant.
\end{corollary}
We remark that when $\frac{h}{2}=1$, $M(\mathcal{S}) = \ML^*(\mathbb Z) = \ML(\mathbb{Z})$ is one dimensional. The coarse asymptotic for $|M\left(\mathcal S\right) \cdot \Y \cap B_R(\X)|$ when $\frac{h}{2}=1$ is separated out as a special case and treated in Corollary \ref{second main corollary}.

The above results says for example, the number of Dehn twist lattice points intersecting a closed ball of radius $R$ in the Teichm\"{u}ller space is coarsely asymptotic to $e^{\frac{h}{2} R}$. Note that any $\X \in \T_{g,n}$ lies in $\T_{g,n}^\epsilon$ for some $\epsilon$, thus another way to phrase the theorem is by picking $\X,\Y \in \T_{g,n}$ first and then by picking any $\epsilon>0$ such that $\X,\Y \in \T_{g,n}^\epsilon$. The constant $J$ and the above results follow. Recall that in the $\epsilon$-thick part of Teichm\"{u}ller space, there is a uniformly bounded difference, depending on $\epsilon$, between the Thurston metric and Teichm\"{u}ller metric \cite{lenzhen2012bounded}. Thus the above results also hold for the Thurston metric after a slight variation.

Our argument hinges on studying how the length of any simple closed geodesic $\tau$ on a hyperbolic structure $\mathcal{X}$ changes after applying a twist $T_\alpha$. To this end, in Theorem \ref{Hyperbolic Length} we obtain an explicit bound on the length of $\ell_{T_\alpha \mathcal{X}}(\tau)$ in terms of $\ell_{\mathcal{X}}(\tau)$, $\ell_{\mathcal{X}}(\alpha)$ and the intersection patterns between $\tau$ and $\alpha$, up to additive error. We then use this Theorem \ref{Hyperbolic Length}, together with results of Choi, Rafi \cite{choi2007comparison} and Lenzhen, Rafi, Tao \cite{lenzhen2012bounded}, to realize a precise relationship between $\ell_{\mathcal{X}}(\alpha)$ and $d_{\mathcal{T}}(\mathcal{X}, T_\alpha \mathcal{X})$. This relation is stated in the following theorem.

\begin{thmn}[\ref{Coarse Distance Formula}]
	Fix some $S_{g,n}$ and given any $\epsilon > 0$, there exists a constant $H > 0$ such that given any $\X \in \T_{g,n}^\epsilon$, we have
	\begin{align*}
		d_\T(\X, T_\alpha \X) \stackrel{+H}{\asymp} \log\left(\sum_{i=1}^{k} |a_i|\ell_\X^2(\alpha_i)\right)
	\end{align*}
	for any $\alpha = \sum_{i=1}^k a_i \alpha_i \in \ML^*(\mathbb Z)$.
\end{thmn}

Moreover, we have constructed a sequence of multicurves in $\ML(\mathbb{Z}) \setminus \ML^*(\mathbb{Z})$ for which Theorem \ref{Coarse Distance Formula} does not hold, see Remark \ref{counterexample to coarse distance formula}. There exists a $H'>0$ depends on $S_{g,n}$ and $\epsilon$, so that for these multicruves the distances behave like
\begin{align*}
	d_\T(\X, T_\alpha \X) \stackrel{+H'}{\asymp} \log\left(\sum_{i=1}^{k} |a_i|\ell_\X(\alpha_i)\right)
\end{align*}
for any $\X \in \T_{g,n}^\epsilon$. This leads to the following question.

\begin{question}
	For $\alpha \in  \ML(\mathbb{Z}) \setminus \ML^*(\mathbb{Z})$, how does the length of any simple closed geodesic $\tau$ on a hyperbolic structure $\mathcal{X}$ changes after applying a twist $T_\alpha$? How far does a point move in Teichm\"{u}ller space after applying the corresponding twist $T_\alpha$?
\end{question}

Another natural question prompted by Theorem \ref{Main Theorem} and Theorem \ref{Coarse Distance Formula} is

\begin{question}
Let $D(\ML(\mathbb{Z}))$ denote the set of all twists. What are the coarse asymptotics for $D(\ML(\mathbb{Z}))$?
\end{question}
Recall that Mirzakhani's Theorem \ref{Counting Formula} says
\begin{align*}
	|\{\alpha \in \ML(\cdot) \mid \ell_X(\alpha) \le e^{\frac{R}{2}}\}| \sim n_X(\cdot) \cdot e^{\frac{h}{2}R},
\end{align*}
which is at the same coarse asymptotic rate of $|D \cdot \Y \cap B_R(\X)|$ for the three cases as in Theorem \ref{Main Theorem} and Corollary \ref{Main Corollary}. Moreover, Mirzakhani \cite{mirzakhani2008growth} also proves that for any $X \in \M_{g,n}$, there exists a constant $n_X$ such that
\begin{align*}
	|\{\alpha \in \ML(\mathbb Z) \mid \ell_X(\alpha) \le e^{\frac{R}{2}}\}| \sim n_X \cdot e^{\frac{h}{2}R}.
\end{align*}
We may wonder whether $|D(\ML(\mathbb{Z})) \cdot \Y \cap B_R(\X)|$ is coarsely asymptotic to $n_X \cdot e^{\frac{h}{2}R}$ as well? This turns out to be false. Namely, we show there is a subset $\ML([\gamma]) \subset \ML(\mathbb Z)$ such that $|D(\ML([\gamma])) \cdot \Y \cap B_R(\X)|$ is at least coarsely asymptotic to $R \cdot e^{\frac{h}{2}R}$, forcing a lower bound for the coarse asymptotic rate of $|D(\ML(\mathbb Z)) \cdot \Y \cap B_R(\X)|$. At the end of section \ref{proof of second main theorem}, we discuss the difficultly using Theorem \ref{Coarse Distance Formula} to obtain an upper bound estimate for the coarse asymptotic rate of $|D(\ML(\mathbb Z)) \cdot \Y \cap B_R(\X)|$.

Let $\underline{\gamma} = \sum_{i=1}^k \gamma_i$ denote a multicurve with all coefficients equal to one and of maximal dimension $k= \frac{h}{2}$. We say $\gamma = \sum_{i=1}^k a_i \gamma_i \in [\underline{\gamma}]$ if $\gamma$ and $\underline{\gamma}$ are the same when without coefficients. Let's denote
\begin{align*}
	\ML\left([\underline{\gamma}]\right) = \bigsqcup_{\gamma \in [\underline{\gamma}]} \ML(\gamma).
\end{align*}
Notice $\ML([\underline{\gamma}])$ consists of infinity many conjugacy classes of multicurves.

\begin{theorem} \label{second main theorem}
	Given any $S_{g,n}$ such that $h > 0$, $\epsilon > 0$, and $\underline{\gamma} = \sum_{i=1}^k \gamma_i$ a multicurve with all coefficients equal to one and of maximal dimension $k= \frac{h}{2}$. There exists a number $f(\underline{\gamma})$ such that, for any $\X, \Y \in \T_{g,n}^\epsilon$, 
	\begin{align*}
	|D(\ML(\mathbb Z)) \cdot \Y \cap B_R(\X)| \ge \left|D\left(\ML([\underline{\gamma}])\right) \cdot \Y \cap B_R(\X)\right| \stackrel{*JF(\X,\Y)}\succeq f(\underline{\gamma}) \cdot R \cdot e^{\frac{h}{2}R}.
	\end{align*}
\end{theorem}
In particular, we can consider the case $\frac{h}{2}=3g-3+n=1$, where $S_{g,n}$ is either $S_{1,1}$ or $S_{0,4}$, and $\Mod_{g,n}, \T_{g,n}$ are $\text{SL}_2(\mathbb{Z}), \mathbb{H}^2$ respectively. In this case, $\ML(\mathbb{Z})$ is one dimensional and we have $\ML(\mathbb{Z}) =\ML([\underline{\gamma}])$ for any simple closed curve $\underline{\gamma}$. In correspondence $D(\ML(\mathbb{Z}))$ is the set of all parabolic elements of $\text{SL}_2(\mathbb{Z})$. There are many results about the asymptotic growth of lattice points in $\mathbb{H}^2$, see \cite{Huber1956berEN}, \cite{parkkonen2015hyperbolic} for example. The corollary below can also be interpreted as a coarse asymptotic for the number of parabolic lattice points of $\text{SL}_2(\mathbb{Z})$ intersecting a closed ball of radius $R$ in $\mathbb{H}^2$.
\begin{corollary} \label{second main corollary}
	Given $S_{g,n}$ equal to $S_{1,1}$ or $S_{0,4}$ and given any $\epsilon > 0$.  For any $\X, \Y \in \T_{g,n}^\epsilon$, we have
	\begin{align*}
	\left| D \left(\mathcal ML(\mathbb{Z})\right) \cdot \Y \cap B_R(\X) \right| \stackrel{*4JF(\X,\Y)}{\sim} n_X(\mathcal S) \cdot R \cdot e^{R}.
	\end{align*}
\end{corollary}
The upper bound in this Corollary follows from an alternation of the proof of Corollary \ref{Main Corollary}, see section \ref{Proof of the Main Theorem}, and the lower bound follows from previous Theorem \ref{second main theorem}, see section \ref{proof of second main theorem}.

We conclude this introduction with two more questions for further study.
\begin{question}
		For a set of twists $D$ with known coarse asymptotics, we may next ask for more precise asymptotics, i.e., what is the best coarse asymptotic coefficient $J$ we can achieve? 
\end{question}

\begin{question}\label{pA}
	In Theorem \hyperref[Main Theorem]{1.4} we are essentially looking at the growth behavior of the orbit of a conjugacy class of twists. What about the asymptotics growth behavior of other conjugacy classes in $\Mod_{g,n}$?
\end{question}
We explore Questions \ref{pA} for pseudo-Anosov conjugacy classes in the following paper \cite{pApaper}.

The organization of the paper is as follows. 
In section \ref{Background} we briefly review some background and previous results.
In section \ref{The effect of twist on hyperbolic length} we study how twists affect the lengths of simple closed curves and prove Theorem \ref{Hyperbolic Length}.
In section \ref{Coarse distance}, we prove Theorem \ref{Coarse Distance Formula} estimating how far a twist translates a point in Teichm\"uller space.
In section \ref{Proof of the Main Theorem}, we prove our main results Theorem \ref{Main Theorem} and Corollary \ref{Main Corollary} using Mirzakhani's result of counting simple closed geodesics \cite{mirzakhani2008growth}. 
In section \ref{proof of second main theorem}, we prove Theorem \ref{second main theorem} and Corollary \ref{second main corollary} quickly follows. 
We remark that Theorem \ref{Hyperbolic Length} and Theorem \ref{Coarse Distance Formula} are key ingredients in our argument and may be of independent interest.

\subsection{Acknowledgments}
I would like to thank my advisor, Spencer Dowdall, for suggesting this project, for his guidance and consistent support throughout. This paper would not exist without his help. I am also grateful to James Farre, Bin Sun, Zach Tripp for heplful conversations and comments. Finally, I would like to thank to Ayelet Lindenstrauss and Kevin Pilgrim for inspiring me to explore my interest in mathematics.

\section{Background} \label{Background}
We refer the reader to \cite{farb2011primer} for more background materials.

\subsection{Mapping Class Group and Dehn twists}\label{Mapping Class Group and Dehn twists}
Let $\text{Homeo}^+_{g,n}$ denote the group of all the orientation-preserving homeomorphisms of $S_{g,n}$ preserving the set of punctures, and let $\text{Homeo}^0_{g,n}$ denote the connected component of the identity. The mapping class group of $S_{g,n}$ is defined to be the group of isotopy classes of orientation-preserving homeomorphisms:
\begin{align*}
\Mod_{g,n} = \text{Homeo}^+_{g,n}/ \text{Homeo}^0_{g,n} = \text{Homeo}^+_{g,n}/\ \text{isotopy} 
\end{align*}

Let $A = S^1 \times [0,1]$ be an oriented annulus, the twist map $T\colon A \to A$ is defined to be $(\theta,t) \mapsto (\theta +2\pi t,t)$, so $T$ is a homeomorphism of $A$ relative to its boundary. Let $a$ be a representative of a simple closed curve $\alpha$ on $S_{g,n}$ and let $N$ be a regular neighborhood of $a$. Pick some orientation-preserving homeomorphism $\phi\colon A \to N$, the Dehn twist about $a$ is defined by
\begin{align*}
T_a(x) =
\begin{cases}
\phi \circ T \circ \phi^{-1}(x) \ \text{if}\ x \in N \\
x \ \text{if}\ x \in S_{g,n} \setminus N 
\end{cases}
.
\end{align*}
The isotopy class of $T_a$ does not depend on choice of $a$ in $\alpha$. Thus $T_\alpha$ is an well-defined mapping class. Now given any multicurve $\alpha = \sum_{i=1}^k a_i \alpha_i$, the composition $T_\alpha = \prod_{i=1}^{k} T_{\alpha_i}^{a_i}$ is called a multi-twist.

Given two simple closed curves $\alpha,\beta$, the intersection number $i(\alpha, \beta)$ is defined to be $i(\alpha, \beta) = \min |a \cap b|$ where $a,b$ are in the isotopy classes $\alpha ,\beta$ respectively and $|a \cap b|$ denotes how many times $a$ and $b$ intersect. The following proposition of Ivanov shows how twists effect intersection numbers.

\begin{proposition}[Intersection Formula \cite{ivanov1992subgroups}]\label{Intersection Formula}
Let $\alpha=\sum_{i=1}^k a_i \alpha_i$ be a multicurve on $S_{g,n}$, and $T_\alpha = \prod_{i=1}^{k} T_{\alpha_i}^{a_i}$ the corresponding twist. Given $\beta, \gamma$ arbitrary simple closed curves on $S_{g,n}$. If $\alpha$ is positive or negative, we have
\begin{align}\label{intersection formula}
\left |i(T_\alpha(\beta), \gamma) - \sum_{i=1}^n |k_i|i(\alpha_i,  \beta)i(\alpha_i, \gamma) \right | \le i(\beta, \gamma).
\end{align}
If $\alpha$ is of mixed sign, we have
\begin{align}\label{signed intersection formula}
&\sum_{i=1}^n (|k_i|-2)i(\alpha_i,  \beta)i(\alpha_i, \gamma) - i(\beta, \gamma) \\ \nonumber
&\le i(T_\alpha(\beta), \gamma) \\ \nonumber
&\le \sum_{i=1}^n |k_i|i(\alpha_i,  \beta)i(\alpha_i, \gamma) + i(\beta, \gamma).
\end{align}
\end{proposition}

\subsection{Teichm\"{u}ller space and moduli space}
A hyperbolic structure $\X$ on $S_{g,n}$ is a pair $(X, \phi)$ where $\phi\colon S_{g,n} \to X$ is a homeomorphism and $X$ is a hyperbolic surface. We say two hyperbolic structures $\X = (X, \phi), \Y = (Y, \psi)$ are isotopic if there is an isometry $I\colon X \to Y$ isotopic to $\psi \circ \phi^{-1}$.   The Teichm\"{u}ller space $\T_{g,n}$ is the set of hyperbolic structures on $S_{g,n}$ modulo isotopy. We let $\X = (X, \phi), \Y = (Y, \psi)$ denote elements in $\T_{g,n}$.

Given any $\X, \Y \in \T_{g,n}$, the Teichm\"{u}ller distance between them is defined to be
\begin{align*}
d_\T(\X, \Y) = \frac{1}{2} \inf_{f \sim \phi \circ \psi^{-1}}\log(K_f)
\end{align*}
where the infimum is over all quasi-conformal homeomorphisms $f$ isotopic to $\phi \circ \psi^{-1}$ and $K_f$ is the quasi-conformal dilatation of $f$. Equipped with the Teichm\"{u}ller metric, the Teichm\"{u}ller space is a complete, unique geodesic metric space.

The mapping class group acts isometrically on $\T_{g,n}$ by changing the marking $(f, (X,\phi)) \mapsto (X,\phi \circ f^{-1})$. This action is properly discontinuous but not cocompact. The quotient $\M_{g,n} = \T_{g,n} / \Mod_{g,n}$ is called the moduli space, and it is a non-compact orbifold parameterizing hyperbolic surfaces homeomorphic to $S_{g,n}$.

Given any $\X = (X, \phi) \in \T_{g,n}$ and given any isotopy class $\gamma$ of nontrivial simple closed curves on $X$, there exists a unique geodesic in this free homotopy class. We define the length function on $X$ by setting $\ell_{X}(\gamma)$ equal to the length of this unique geodesic. We also let $\ell_{\X}(\alpha)$ denote $\ell_{X}(\phi(\alpha))$ for any simple closed curve $\alpha$ on $S_{g,n}$. For any multicurve $\alpha = \sum_{i=1}^k a_i\alpha_i$, we define $\ell_{\X}(\alpha) = \sum_{i=1}^k |a_i|\ell_{\X}(\alpha_i)$ to be its length.

A pair of pants is a closed surface of zero genus with three boundary components or punctures. A pants decomposition $\Gamma$ of the surface $S_{g,n}$ is a collection of pairwise disjoint non-trivial simple closed curves $\gamma_1, \cdots, \gamma_{3g-3+n}$ on $S_{g,n}$, together they decompose the surface $S_{g,n}$ into $2g+n-2$ pairs of pants. Using pants decomposition and by introducing Fenchel-Nielsen coordinates, Fricke \cite{FrickeKlein} showed that the dimension of $\T_{g,n}$ is $6g+2n-6$.

A theorem of Bers \cite{Bers1985} says there exists a constant depending only on $S_{g,n}$ such that for every $\X \in \T_{g,n}$, there is a pants decomposition $\Gamma_\X$ of $\X$ in which each simple closed curve has length bounded above by this Bers' constant.

Given any $\epsilon > 0$, the $\epsilon$-thick part of Teichm\"{u}ller space is defined to be
\begin{align*}
\T^\epsilon_{g,n} = \{\X \in \T_{g,n} \mid \ell_\X(\alpha) \ge \epsilon \text{\ for any simple closed curve\ } \alpha \text{\ on\ } S_{g,n}\}
\end{align*}
and consequently the $\epsilon$-thick part of moduli space is $\M^\epsilon_{g,n} = \T_{g,n}^\epsilon / \Mod_{g,n}$. The Mumford compactness criterion \cite{10.2307/2037802} says $\M^\epsilon_{g,n}$ is compact for any $\epsilon >0$.
\subsection{Short Marking}\label{Short Marking}
For any $\X \in \T_{g,n}$, a short marking \cite{choi2007comparison} $\mu_\X$ is a collection of simple closed curves $\{\eta_i\}_{i=1}^{3g-3+n} \cup \{\delta_i\}_{i=1}^{3g-3+n}$ on $S_{g,n}$ picked in the following way: First, choose a pant decomposition $\{\eta_i\}_{i=1}^{3g-3+n}$ by taking a curve $\eta_1$ on $S_{g,n}$ that is a shortest curve with respect to $\X$, and then a next shortest disjoint curve from the first, and so on until we complete a pants decomposition. Next, for each $\eta_i$, pick a shortest curve $\delta_{i}$ that intersects $\eta_{i}$ and is disjoint from all other pants curves. For each $i$, we say $\eta_{i}, \delta_{i}$ is a pair. The collection of curves obtained in this way has the property that
any two curves have intersection number bounded by $2$. Note there could be a finite number of possible short markings corresponding to each $\X \in \T_{g,n}$, we fix one such short marking and call it the short marking $\mu_\X$. Moreover, given any $\epsilon > 0$, by Bers' Theorem and trigonometry, there exists $N > 0$ depending on $\epsilon$ and $S_{g,n}$ such that for any $\X \in \T_{g,n}^\epsilon$, all curves in the short marking $\mu_\X$ have length bounded above by $N$ and bounded below by $\epsilon$. 

We recall a result from Choi and Rafi \cite{choi2007comparison} stating that for any $\epsilon > 0$, the Teichm\"{u}ller distance in the $\epsilon$-thick part can be approximated by the maximum ratio of change of lengths of the short marking.

\begin{theorem}[Distance Formula \cite{choi2007comparison}]\label{Distance Formula}
	For any $\epsilon > 0$, there exists $c > 0$ depending on $S_{g,n}$ and $\epsilon$ such that for any $\X, \Y \in \T_{g,n}^\epsilon$
	\begin{align*}
	d_\T(\X, \Y) \stackrel{+c}{\asymp} \log\max_{\gamma \in \mu_ \X}\frac{\ell_\Y(\gamma)}{\ell_\X(\gamma)}.
	\end{align*}
\end{theorem}
We also recall that Lenzhen, Rafi, Tao \cite{lenzhen2012bounded} showed that for any simple closed curve on $S_{g,n}$, its length with respect to $\X$ can be estimated via its intersection pattern with the short marking $\mu_\X$.
\begin{proposition}[Length Formula \cite{lenzhen2012bounded}]\label{Length Formula}
	There exists $C \ge 1$ depending on $S_{g,n}$ such that for any simple closed curve $\beta$ on $S_{g,n}$ and for any $\X \in \T_{g,n}$, we have
	\begin{align*}
	\ell_\X(\beta) \stackrel{*C}{\asymp} \sum_{\gamma \in \mu_\X} i(\beta, \gamma) \ell_\X(\bar{\gamma})
	\end{align*}
	where $\bar{\gamma}$ denotes the curve in the short marking paired with $\gamma$. 
\end{proposition}

For a fixed $\epsilon$, any curve $\gamma$ in $\mu_\X, \X \in \T_{g,n}^\epsilon$ satisfies $\epsilon \le \ell_\X(\gamma) \le N$. We can therefore rewrite the above theorem and proposition for $\T_{g,n}^\epsilon$.

\begin{lemma}
	For any $S_{g,n}$ and $\epsilon > 0$, there exists $C$ depends on $S_{g,n}$ and $c, N$ depends on $S_{g,n}$ and $\epsilon$ such that
	\begin{align}\label{distance formula}
		\log \left( \frac{1}{Ne^c} \max_{\gamma \in \mu_\X}\ell_\Y(\gamma) \right) \le d_\T(\X,\Y) \le \log \left( \frac{e^c}{\epsilon} \max_{\gamma \in \mu_\X}\ell_\Y(\gamma) \right)
	\end{align}
	\begin{align}\label{length formula}
	\frac{\epsilon}{C} \sum_{\gamma \in \mu_\X} i(\beta, \gamma) \le \ell_\X(\beta) \le CN\sum_{\gamma \in \mu_\X} i(\beta, \gamma)
	\end{align}
\end{lemma}

\begin{figure}[h]
	\begin{center}
		\begin{tikzpicture}
		\node[anchor=south west,inner sep=0] at (0,0) {\includegraphics[scale=0.4]{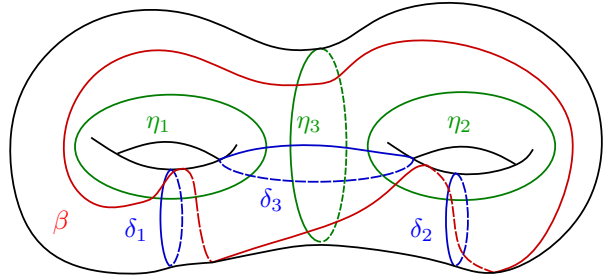}};
		\node at (2, 2){\textcolor{OliveGreen}{$\eta_1$}};
		\node at (6, 2){\textcolor{OliveGreen}{$\eta_2$}};
		\node at (4, 2){\textcolor{OliveGreen}{$\eta_3$}};
		\node at (1.7, .6){\textcolor{blue}{$\delta_1$}};
		\node at (5.5, .6){\textcolor{blue}{$\delta_2$}};
		\node at (3.5, 1){\textcolor{blue}{$\delta_3$}};
		\node at (.7, .7){\textcolor{red}{$\beta$}};
		\end{tikzpicture}	
	\end{center}
	\caption{A short marking $\mu_\X$ and a simple closed curve $\beta$ on a hyperbolic surface $X$ homeomorphic to $S_2$.}
	\label{pic1}
\end{figure}

\subsection{Properties of $\mathbb{H}^2$} \label{upper half plane properties}

We first recall the following useful lemma in hyperbolic geometry.
\begin{lemma}[Collar Lemma \cite{farb2011primer}] \label{Collar Lemma}
	For any simple closed geodesic $\gamma$ of length $\ell$ on a hyperbolic surface, it is contained in an embedded cylinder of diameter of order $\ell^{-1}$, and the diameter is
	\begin{align*}
		W\left(\gamma\right) = \sinh^{-1}\left(\frac{1}{\sinh(\frac{1}{2}\ell)}\right)
	\end{align*}
\end{lemma}

For any two closed sets $A,B \subset \mathbb{H}^2$ we let $d(A,B)$ denote the minimal distance between them. For any geodesic $\eta$ in $\mathbb{H}^2$, we let $\pi_\eta$ denote the closest point projection map, namely
\begin{align*}
\pi_\eta(x) = \{y \in \eta \mid d(x,y) = d(x,\eta)\}.
\end{align*} 
For any two points $x, y \in \mathbb{H}^2$, we let $[x,y]$ denote the unique geodesic connecting them. Given two points $x,y \in \mathbb{H}^2$ separated by a bi-infinite geodesic $\eta$ and far away from $\eta$, we let $x_\eta \in \mathbb{H}^2$ denote the first point that the geodesic $[x,y]$ enters the $L$-neighborhood of $\eta$ coming from the $x$ side. Similarly, we can define $y_\eta \in \mathbb{H}^2$. If $x$ is in the $L$-neighborhood of $\eta$ to begin with, we just let $x_\eta = x$, and similarly for $y$.

Being a hyperbolic space, geodesics are strongly contracting in $\mathbb{H}^2$, see \cite{Arzhantseva_2015} for example. That is, there exists a constant $L$ such that for any geodesic $\eta$ and for any geodesic $\alpha$ that never enters the $L$-neighborhood of $\eta$, the diameter of $\pi_\eta(\alpha)$ is bounded by $L$. As a consequence, we have
\begin{corollary}\label{Crossing Fellow Travel}
There exists a constant $L$ such that for any bi-infinite geodesic $\eta$ in $\mathbb{H}^2$ and for any two points $x,y$ separated by $\eta$, we have
	\begin{align*}
	d(x_\eta, \pi_\eta(x)) \le 2L, d(y_\eta, \pi_\eta(y)) \le 2L.
	\end{align*}
\end{corollary}
This is because we have $d(x_\eta, \pi_\eta(x)) \le d(x_\eta, \pi_\eta(x_\eta)) + d(\pi_\eta(x_\eta), \pi_\eta(x)) \le 2L$. Similarly for $y_\eta$.

Another important property of the projection map in $\mathbb{H}^2$ is that it's $1$-Lipschitz. Viewing in the upper half plane model and up to isometry, we may assume $\eta$ is the vertical line $x=0$. For each point $(0,r) \in \eta$, the points projecting to $(0,r)$ are exactly the Euclidean semicircles of radius $r$ centered at $(0,0)$. Given two Euclidean semicircle centered at $(0,0)$, the minimal distance between them are realized by the points intersecting the vertical line $x=0$. This means
\begin{lemma}\label{1-Lip}
$\pi_\eta$ is $1$-Lipschitz for any bi-infinite geodesic $\eta$ in $\mathbb{H}^2$.
\end{lemma}
\subsection{Lifts of twists} \label{Lifts of twists}
Given an oriented bi-infinite geodesic $\beta$ in $\mathbb{H}^2$ and a number $l_\beta \in \mathbb{R}$, we can decompose $\mathbb{H}^2$ into two open pieces, one to the left of $\beta$ and one to the right of $\beta$, and then regule the two pieces along $\beta$ after translation according to $l_\beta$. When $l_\beta$ is positive, we regule the pieces along $\beta$ after translating distance $|l_\beta|$ to the left. When $l_\beta$ is negative, we regule the pieces along $\beta$ after translating distance $|l_\beta|$ to the right. This process is called shearing along $\beta$ according to $l_\beta$, see \cite{kerckhoff1980} for more detail. We are mainly interested in what happens to geodesics after shearing. Let $\tau$ be a bi-infinite geodesic in $\mathbb{H}^2$ transverse to $\beta$ and let $\tau'$ be the image of $\tau$ after shearing along $\beta$ according to $l_\beta$, then $\tau'$ is a concatenation of two geodesic rays with a sub-segment of $\beta$ of length $l_\beta$ connecting these two rays' starting points, see Figure \ref{pic2} for an illustration.

\begin{figure}[h]
	\begin{center}
		\begin{tikzpicture}
		\node[anchor=south west,inner sep=0] at (0,0) {\includegraphics[scale=0.25]{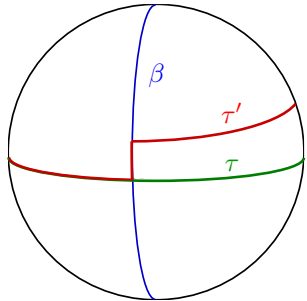}};
		\node at (2, 3){\textcolor{blue}{$\beta$}};
		\node at (3, 1.8){\textcolor{OliveGreen}{$\tau$}};
		\node at (3, 2.5){\textcolor{red}{$\tau'$}};
		\end{tikzpicture}	
	\end{center}
	\caption{After shearing along $\beta$ according to $l_\beta$, $\tau$ becomes $\tau'$.}
	\label{pic2}
\end{figure}

Given $\X = (X, \phi) \in \T_{g,n}$ and let $p\colon \mathbb{H}^2 \to X$ be the universal cover. For any multicurve $\alpha = \sum_{i=1}^k a_i \alpha_i$, we let $\A = \{\alpha_i\}_{i=1}^k$ and let $\tilde{\A}$ denote the set of lifts of curves in $\A$.  For each curve $\tilde{\alpha} \in \tilde{\A}$, we let $\alpha_{s(\tilde{\alpha})}$ denote the curve such that $\tilde{\alpha}$ is a lift of $\alpha_{s(\tilde{\alpha})}$. Note the complements of $\cup_{\tilde{\alpha} \in \tilde{\A}} \tilde{\alpha}$ are infinitely many open regions. Fixing one of these regions, we can shear along all these bi-infinite geodesics in $\tilde{\A}$ according to $a_{s(\tilde{a})} \ell_\X(\alpha_{s(\tilde{\alpha})})$, and this is called shearing according to $\alpha$.

Now, given any simple closed geodesic $\beta$ on $\X$, we let $\tau$ be a lift of $\beta$ with a base point $q_0 \in \tau$. Fixing the region containing $q_0$, we can shear according to $\alpha$. Let $\tau'$ denote the image of $\tau$ after shearing, then the projection of $\tau'$ is isotopic to the simple closed geodesic $T_\alpha(\beta)$. Let $q_L', q_R' \in \partial \mathbb{H}^2$ denote the endpoints of $\tau'$. The two end points $q_L', q_R' \in \partial \mathbb{H}^2$ define a unique bi-infinite geodesic $\sigma$ in $\mathbb{H}^2$ and $\sigma$ is in the same isotopy class of $\tau'$, see Figure \ref{pic3}. This means $\sigma$ is a lift of the simple closed geodesic $p(\sigma) = T_\alpha(\beta)$. Similarly, one can obtain the simple closed geodesic $T_\alpha^{-1}(\beta)$ by shearing in the opposite direction.

\begin{figure}[h]
	\begin{center}
		\begin{tikzpicture}
		\node[anchor=south west,inner sep=0] at (0,0) {\includegraphics[scale=0.35]{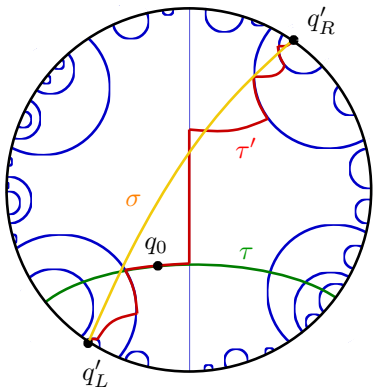}};
		\node at (3.5, 1.6){\textcolor{OliveGreen}{$\tau$}};
		\node at (3.5, 3){\textcolor{red}{$\tau'$}};
		\node at (2, 2.3){\textcolor{orange}{$\sigma$}};
		\node at (2.3, 1.7){$q_0$};
		\node at (4.5, 4.7){$q'_R$};
		\node at (1.5, 0){$q'_L$};
		\end{tikzpicture}	
	\end{center}
	\caption{After shearing according to $\alpha$ (blue curves are in $\tilde{A}$), the geodesic $\tau$ becomes $\tau'$, and the geodesic $\sigma$ is uniquely defined by the endpoints of $\tau'$}
	\label{pic3}
\end{figure}

\subsection{Bass-Serre Tree} \label{Bass-Serre Tree}
We briefly explain how to construct a Bass-Serre tree dual to an infinite collection of bi-infinite geodesics in $\mathbb{H}^2$ that arise from a covering map. In particular, one may imagine how to construct a Bass-Serre dual to the Figure \ref{pic3}. See \cite{scott1979topological} for more detail about Bass-Serre tress in general.

Let $p \colon \mathbb{H}^2 \to S_{g,n}$ be a universal cover. Given $\A = \{\alpha_i\}_{i=1}^n$ a collection of disjoint simple closed curves on $S_{g,n}$, we let $\tilde{\A}$ denote the set of all liftings of curves in $\A$ to $\mathbb{H}^2$, and we let $\cup \tilde{\A}$ denote the union of all elements in $\tilde{\A}$. Define $\Z_\A$ to be the tree dual to $\tilde{\A}$ in $\mathbb{H}^2$. That is, $\Z_\A = (V_\A, E_\A)$ is a graph such that each vertex in $V_\A$ corresponds to a connected component in $\mathbb{H}^2 \setminus \cup \tilde{\A}$ and each edge is dual to an element in $\tilde{\A}$. We label each edge by the element in $\tilde{\A}$ that it is dual to.

Denote the connected component corresponding to a vertex $v$ as $C(v)$. Given two vertices $v,w \in V_\A$, $(v,w) \in E_\A$ if and only if $C(v),C(w)$ represent bordered connected components. Denote $d_\Z$ the metric on the tree $\Z_\A$ where the length of each edge has length 1, $(\Z_\A, d_\Z)$ is a unique geodesic metric space. 

By the Collar Lemma \ref{Collar Lemma}, there exists a $r = \min\{W(\alpha_i)\}_{i=1}^n$ sufficiently small such that for any curve $\alpha \in \A$, $\mathcal{N}_r(\alpha)$ is an open annulus. We can define a $\pi_1$-equivariant, continuous and surjective map $\phi_\A\colon \mathbb{H}^2 \to \mathcal{Z}_\A$ such that each $\mathcal{N}_r(\tilde{\alpha})$ maps to an edge and each connected component in $\mathbb{H}^2 \setminus \cup_{\tilde{\alpha} \in \tilde{\A}} \N_r(\tilde{\alpha})$ gets mapped to a vertex.

Now, given any simple closed curve $\tau$ on $S_{g,n}$ and let $\tilde{\tau}$ be a lift of $\tau$ in $\mathbb{H}^2$. If $\tau$ does not intersect any curve in $\A$, then $\phi_\A(\tilde{\tau})$ is a vertex. Otherwise, denote
\begin{align*}
     i(\tau, \A) = \sum_{i=1}^n i(\tau, \alpha_i)
\end{align*}
the intersection number of $\tau$ with curves in $\A$, $\phi_\A(\tilde{\tau})$ is a bi-infinite geodesic in $(\Z_\A, d_\Z)$. The hyperbolic isometry of $\mathbb{H}^2$ along $\tilde{\tau}$, with translation distance equals to the length of $\tau$, is equivariant with respect to $\phi_\A$ and gives rise to an isometry $\rho_{\tilde{\tau}}$ of $(\Z_\A, d_\Z)$ with translation length $i(\A, \tau)$ and translation axis $\phi_\A(\tilde{\tau})$. This means for any vertex $s$ on the axis $\phi_\A(\tilde{\tau})$, we have $d_\Z(s, \rho_{\tilde{\tau}}(s)) = i(\tau, \A)$.

\subsection{Counting Simple Closed Geodesics} \label{Counting Simple Closed Geodesics}

Given $\gamma$ a simple closed curve or multicurve on any $X \in \M_{g,n}$, we denote
\begin{align*}
s_X(L, \gamma) = |\{\alpha \in \Mod_{g,n} \cdot \gamma \mid \ell_X(\alpha) \le L \}|
\end{align*}
the number of simple closed geodesics on $X$ of topological type $\gamma$ and of hyperbolic length at most $L$. The following formula is due to Mirzakhani.
\begin{theorem}[Counting Formula \cite{mirzakhani2008growth}]\label{Counting Formula}
Fix some $S_{g,n}$, given $\gamma$ a simple closed curve or a multicurve on any $X \in \M_{g,n}$, we have
	\begin{align}\label{counting formula}
	s_X(L, \gamma) \sim n_X(\gamma) \cdot L^{6g+2n-6}
	\end{align}
	where $n_X( \gamma)$ depends on the hyperbolic structure $X$ and the topological type of $\gamma$.
\end{theorem}
Later in the paper, we will count the sum of several topological types of multicurves. Thus we phrase the above Theorem \ref{Counting Formula} in the following equivalent way.
\begin{remark} \label{Repharse Counting Formula}
	  For any $\gamma, X$ and $\lambda >1$, there exist constants $n_X(\gamma)$ and $r_X(\gamma, \lambda)$ such that 
	\begin{align} \label{repharse counting formula}
		\frac{1}{\lambda} \cdot n_X( \gamma) \cdot L^{6g+2n-6} \le s_X(L, \gamma) \le \lambda \cdot n_X(\gamma) \cdot L^{6g+2n-6}
	\end{align}
	for any $L \ge r_X(\gamma, \lambda)$.
\end{remark}
It's also necessary for us to know how $n_X(\gamma)$ and $r_X(\gamma, \lambda)$ behave with respect to scaling the curve $\gamma$ for later purposes.
\begin{corollary} \label{nr}
		For any $\gamma, X, \lambda >1$ and $c \in \mathbb{N}$, we have
		\begin{align*}
		&r_X( c \cdot \gamma, \lambda) = c \cdot r_X(\gamma, \lambda) \\
		&n_X(c \cdot \gamma) = \frac{n_X(\gamma)}{c^{6g+2n-6}}
		\end{align*}
\end{corollary}
\begin{proof}
	Indeed, we have
	\begin{align*}
	s_X(L, c \cdot \gamma) = s_X \left( \frac{L}{c}, \gamma \right)
	\end{align*}
	and that
	\begin{align*}
	\frac{1}{\lambda} \cdot \frac{n_X(\gamma)}{c^{6g+2n-6}} \cdot L^{6g+2n-6} \le s_X \left(\frac{L}{c}, \gamma \right) \le \lambda \cdot \frac{n_X(\gamma)}{c^{6g+2n-6}} \cdot L^{6g+2n-6}
	\end{align*}
for any $\frac{L}{c} \ge r_X( \gamma, \lambda)$. This gives us the desired result.
\end{proof}

Since there are only finitely many topological types of simple closed curves, we denote $n_X( \mathcal S)$ the finite sum of $n_X( \gamma)$ where $\gamma$ ranges over all topological types of simple closed curves on $S_{g,n}$. We will use the notation $s_X(L, \mathcal S)$ to denote the number of all simple closed geodesics that have length bounded by $L$, and we will denote $r_X(\mathcal S, \lambda) = \max_{\gamma \in \mathcal S} r_X(\gamma, \lambda)$ for any $\lambda >1$.

\section{The effect of twisting on hyperbolic length} \label{The effect of twist on hyperbolic length}
In this section, we study how the length of simple closed geodesics on a hyperbolic surface change after applying a twist. In the next section, we use the results below to estimate how far a point in Teichm\"{u}ller space moves after applying a twist.

As our first result, we may obtain the following estimate from the length formula (\ref{length formula}) and intersection formula (\ref{intersection formula}).
\begin{proposition} \label{hyperbolic length not using}
	Fix some $\epsilon > 0$. Given a multicurve $\alpha = \sum_{i=1}^k a_i \alpha_i$ and a simple closed curve $\tau$ on a hyperbolic surface $\X \in \T_{g,n}^\epsilon$, there exists a constant $A$ depends only on $S_{g,n}$ and $\epsilon$ such that
	\begin{align*}
	 A \left(\sum_{i=1}^k |a_i| i(\alpha_i, \tau) \ell_\X(\alpha_i)+\ell_\X(\tau)\right) \ge \ell_{T_\alpha\X}(\tau) \ge \frac{1}{A}\left(\sum_{i=1}^k (|a_i|-2) i(\alpha_i, \tau) \ell_\X(\alpha_i)-\ell_\X(\tau)\right).
	\end{align*}
	Furthermore, if $\alpha$ is positive or negative, the lower bound can be sharpened to
	\begin{align*}
	\ell_{T_\alpha\X}(\tau) \ge \frac{1}{A}\left(\sum_{i=1}^k |a_i| i(\alpha_i, \tau) \ell_\X(\alpha_i)-\ell_\X(\tau)\right).
	\end{align*}
\end{proposition}
\begin{proof}
	By the length formula (\ref{length formula}), we know
	\begin{align*}
	 CN\sum_{\gamma \in \mu_X} i(T_\alpha^{-1}(\tau), \gamma) \ge \ell_{T_\alpha\X}(\tau)  \ge \frac{\epsilon}{C}\sum_{\gamma \in \mu_X} i(T_\alpha^{-1}(\tau), \gamma).
	\end{align*}
	We can apply the signed intersection formula (\ref{signed intersection formula}) to approximate $i(T_\alpha^{-1}(\tau), \gamma)$. This allows us to expand the above inequality into the following:
	\begin{align*}
	&CN\sum_{\gamma \in \mu_\X} \left( \sum_{i=1}^k |a_i| i(\alpha_i, \gamma) i(\alpha_i, \tau)+i(\tau, \gamma) \right)\\
	&\ge \ell_{T_\alpha\X}(\tau) \\
	&\ge\frac{\epsilon}{C}\sum_{\gamma \in \mu_\X} \left( \sum_{i=1}^k (|a_i|-2) i(\alpha_i, \gamma) i(\alpha_i, \tau)- i(\tau, \gamma) \right).
	\end{align*}
	By switching the order of summations $\sum_{\gamma \in \mu_\X}$ and $\sum_{i=1}^k$ and by applying the length formula (\ref{length formula}) again, we obtain the result in the proposition. 
	
	If $\alpha$ is positive or negative, we can use the intersection formula (\ref{intersection formula}), and going through the same proof give us the sharpened lower bound.
\end{proof}

Note the above proposition provides a good estimate for the length of multicurves up to a multiplicative error. This error arises from our repeated use of length formula (\ref{length formula}). Below we propose a more generalized result that leads to removing this multiplicative error. Let $\lfloor \cdot \rfloor_0$ denote the $0$ threshold function.

\begin{theorem}\label{Hyperbolic Length}
	Given a multicurve $\alpha = \sum_{i=1}^k a_i \alpha_i$ and a simple closed curve $\tau$ on any hyperbolic structure $\X$, we have
	\begin{align}
	&\ell_{\X}(\tau) + \sum_{i=1}^k i(\tau, \alpha_i)|a_i| \ell_\X(\alpha_i)\label{estimate hyerbolic length} \\
	&\ge \ell_{T_\alpha \X}(\tau) \nonumber\\
	&\ge \sum_{i=1}^k i(\tau, \alpha_i) \cdot \Bigl\lfloor{(|a_i|-2)\cdot \ell_{\X}(\alpha_i)-2\ell_\X(\tau) - L\Bigl\rfloor}_0 \nonumber
	\end{align}
	where $L$ is a constant that depends on $\mathbb{H}^2$.
\end{theorem}
\begin{proof}
	Fix the hyperbolic structure we may assume curves are geodesics. Given a multicurve $\alpha = \sum_{i=1}^k a_i \alpha_i$, we denote $\A = \{\alpha_i\}_{i=1}^k$ and denote $\tilde{\A}$ the set of all liftings of curves in $\A$ to $\mathbb{H}^2$. Let $\Z_\A$ denote the corresponding Bass-Serre tree, see section \ref{Bass-Serre Tree}.
	
	For each $\beta \in \tilde{\A}$, we denote $\psi_{\beta} \in \pi_1(S_{g,n})$ the corresponding hyperbolic isometry in $\mathbb{H}^2$. If $\beta, \gamma \in \tilde{\A}$ are lifts of the same $\alpha \in \A$, then $\psi_{\beta}, \psi_{\gamma}$ are conjugate to each other and have the same translation distance equal to $\ell_\X(\alpha)$. This also means there exists an isometry $\psi$ in $\mathbb{H}^2$ that sends $\gamma$ to $\beta$. We can choose this isometry up to composing with any power of $\psi_\beta$ or pre-composing with any power of $\psi_\gamma$. In particular, suppose there are geodesic segments $\beta' \subset \beta, \gamma' \subset \gamma$ such that their length are less than $\ell_\X(\alpha)$, we can choose the isometry $\psi$ in a way such that $\beta'$ and $\psi(\gamma')$ both lie on $\beta$ and intersect the same fundamental domain of the action of $\psi_\beta$.
	
	Given a simple closed curve $\tau$, we denote
	\begin{align*}
		i(\tau, \alpha) = m = \sum_{i=1}^k i(\tau, \alpha_i) = \sum_{i=1}^k m_i.
	\end{align*}
	In the case of $i(\tau, \alpha) = 0$, $T_\alpha$ has no effect on $\tau$ and the theorem holds true. We may assume $i(\tau, \alpha) \ge 1$. Let $\tilde{\tau}$ be a lifting of $\tau$ and say it has end points $q_L, q_R \in \partial \mathbb{H}^2$. $\tilde{\tau}$ is therefore a bi-infinite geodesic in $\mathbb{H}^2$ transverse to $\tilde{\A}$, and $\phi_\A(\tilde{\tau})$ is a bi-infinite geodesic in $\Z_\A$, say the edges are labeled by
	\begin{align*}
		\B = \{\cdots, \beta_{-2}, \beta_{-1}, \beta_0, \beta_1, \beta_2, \cdots\}.
	\end{align*}
	
	For each $\beta_i$, let's denote $\pi_{\beta_i}(\beta_{i-1}), \pi_{\beta_i}(\beta_{i+1})$ as $\beta_{i,L}, \beta_{i,R}$ respectively. Define a index function $s\colon \mathbb{N} \to \{1, \cdots, k\}$ so that $\alpha_{s(i)} \in \A$ is the simple closed curve such that $\beta_i$ is a lift of $\alpha_{s(i)}$. For each $i$, we claim $d(\beta_{i,L}, \beta_{i,R}) \le 2\ell_\X(\tau)$. In the case of $i(\tau, \A) =1$, we pick $\kappa \subset \tilde{\tau}$ to be the geodesic segment between the points $\tilde{\tau} \cap \beta_{i-1}$ and $\tilde{\tau} \cap \beta_{i+1}$. Then $\kappa$ is a concatenation of two consecutive path liftings of $\tau$ and $\ell_\X(\kappa) = 2\ell_\X(\tau)$. For $i(\tau, \A) \ge 2$, we pick $\kappa$ to be the path lifting of $\tau$ starting from $\beta_{i-1} \cap \tilde{\tau}$. In any case, $\kappa \subset \tilde{\tau}$ goes through $\beta_{i-1}, \beta_i, \beta_{i+1}$ and has length bounded by $2\ell_\X(\tau)$. By Lemma \ref{1-Lip} we know the projection maps $\pi_{\beta_i}$ are $1$-Lipschitz, thus the distance between projections of the two endpoints of $\kappa$ on $\beta_i$ is smaller than the length of $\kappa$, which is less than $2\ell_\X(\tau)$. Since the projections of the two endpoints lie in $\beta_{i,L}, \beta_{i,R}$ respectively, we have $d(\beta_{i,L}, \beta_{i,R}) \le 2\ell_\X(\tau)$.
	
	Fix some point $q_0 \in \tilde{\tau}$ and let $\tilde{\tau}'$ be $\tilde{\tau}$ after shearing according to $-\alpha$ fixing the component of $q_0$, see section \ref{Lifts of twists}. The projection of $\tilde{\tau}'$ to the surface $\X$ has length equal to $\ell_{\X}(\tau)+\sum_{i=1}^k i(\tau, \alpha_i)|a_i| \ell_\X(\alpha_i)$. Denote the end points of $\tilde{\tau}'$ as $q_L', q_R' \in \partial \mathbb{H}^2$. Let $\sigma$ be the geodesic with end points $q_L', q_R' \in \partial \mathbb{H}^2$, then $\sigma$ is a lift of the geodesic $T^{-1}_\alpha(\tau)$ and its image $\phi_\A(\sigma)$ is a geodesic in $\Z_\A$. Since the projection of $\tilde{\tau}'$ is in the isotopy class $T^{-1}_\alpha(\tau)$, the upper bound in (\ref{estimate hyerbolic length}) follows.
	
	Once $\tilde{\tau}'$ leaves a connected component of $\mathbb{H}^2 \setminus \cup \tilde{\A}$, it never comes back. This means $\phi_\A(\tilde{\tau}')$ does not back track in $\Z_\A$ so $\phi_\A(\tilde{\tau}')$ is a geodesic path in $\Z_\A$. Since $\sigma$ shares the same endpoints with $\tilde{\tau}'$ and since $\Z_\A$ is a unique geodesic space, we have $\phi_\A(\tilde{\tau}')= \phi_\A(\sigma)$. Denote the edge labels of $\phi_\A(\sigma)$ as
	\begin{align*}
		 \F = \{\cdots, \eta_{-2}, \eta_{-1}, \eta_0, \eta_1, \eta_2, \cdots\}.
	\end{align*}
	
	For each $\eta_i$, let's denote $\pi_{\eta_i}(\eta_{i-1}), \pi_{\eta_i}(\eta_{i+1})$ as $\eta_{i,L}, \eta_{i,R}$ respectively. Since each $\beta_i$ and $\eta_i$ are lifts of the same curve, we can use the same index function $s$ denoting $\alpha_{s(i)} \in \A$ the simple closed curve such that $\eta_i$ is a lift of $\alpha_{s(i)}$. Since for each $i$, the triples $(\tilde{\tau},\beta_i, \beta_{i+1})$ and $(\tilde{\tau}', \eta_i, \eta_{i+1})$ realize the same intersection pattern on the surface, $\eta_{i,L}, \eta_{i,R}$ are translations of $\beta_{i,L}, \beta_{i,R}$ respectively and have the same diameters respectively.
	
	The relative location of $\eta_{i,L}, \eta_{i,R}$ is the same as the relative location of $\beta_{i,L}$ and $\psi_{\beta_i}^{a_{s(i)}}(\beta_{i,R})$, see Figure \ref{pic4} for an illustration. Recall for any point $x$ on any $\beta_i$ and for any $t \in \mathbb{Z}$, we have $d(x,\psi_{\beta_i}^t (x)) = |t|\ell_{\X}(\alpha_{s(i)})$. Since both $\diam(\beta_{i,L}), \diam(\beta_{i,R})$ are bounded by $\ell_{\X}(\alpha_{s(i)})$, and since $d(\beta_{i,L}, \beta_{i,R}) \le 2\ell_\X(\tau)$, we have $\diam(\beta_{i,L} \cup \beta_{i,R}) \le 2\ell_{\X}(\alpha_{s(i)}) + 2\ell_\X(\tau)$. It follows that
	\begin{align*}
	d(\eta_{i,L}, \eta_{i,R}) &\ge|a_{s(i)}| \ell_{\X}(\alpha_{s(i)}) - \diam(\beta_{i,L} \cup \beta_{i,R}) \\
	&\ge (|a_{s(i)}|-2) \ell_{\X}(\alpha_{s(i)}) - 2\ell_\X(\tau).
	\end{align*}
	Denote
	\begin{align*}
	D_i =  (|a_{s(i)}|-2) \ell_{\X}(\alpha_{s(i)}) - 2\ell_\X(\tau)
	\end{align*}
	so that $d(\eta_{i,L}, \eta_{i,R}) \ge D_i$ for any $i$.

	\begin{figure}[h]
		\begin{center}
			\begin{tikzpicture}
			\node[anchor=south west,inner sep=0] at (0,0) {\includegraphics[scale=0.35]{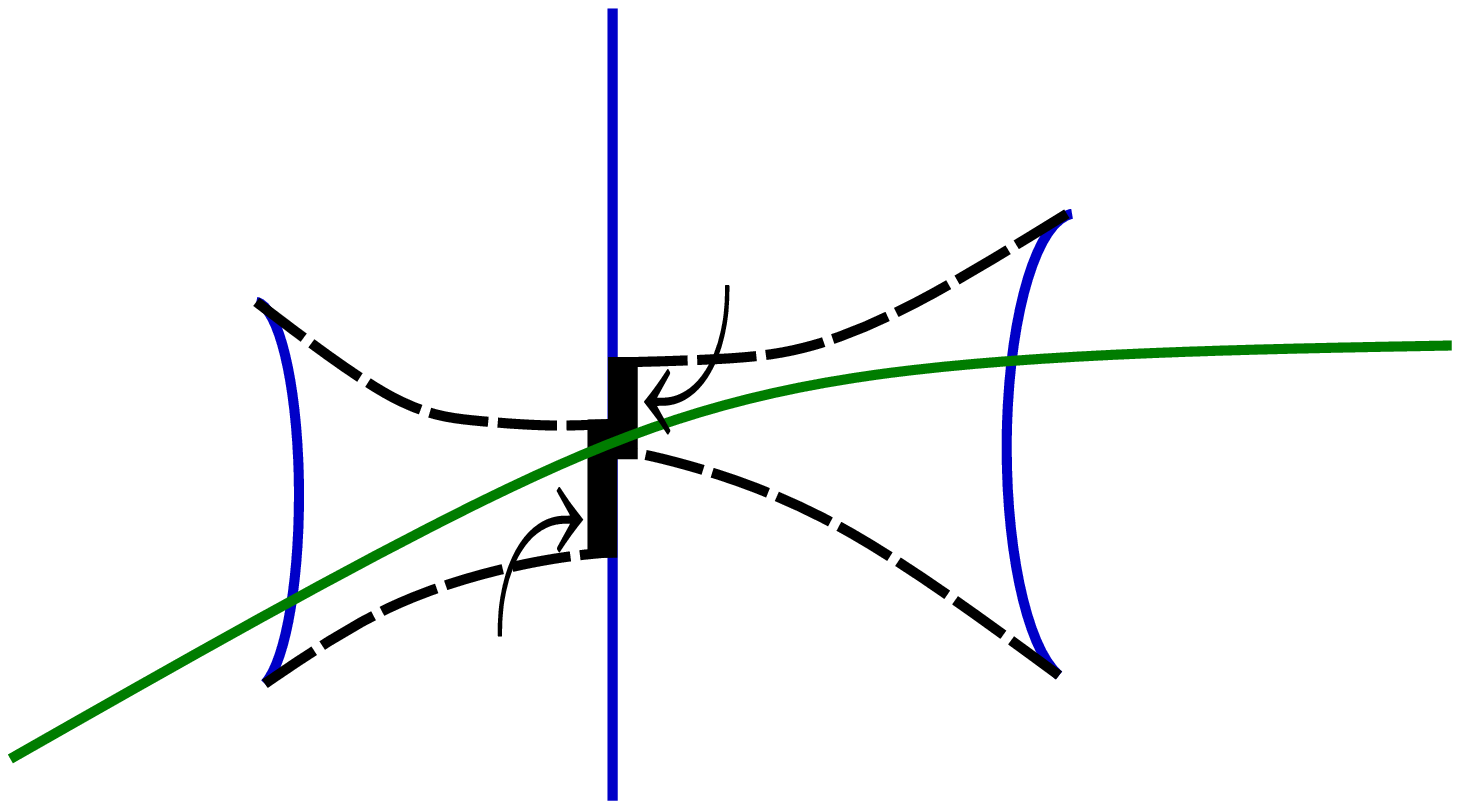}};
			\node at (4.5, 2){\textcolor{OliveGreen}{$\tau$}};
			\node at (1.8, 2.5){\textcolor{blue}{$\beta_i$}};
			\node at (1, 2.5){\textcolor{blue}{$\beta_{i-1}$}};
			\node at (3.8, 2.5){\textcolor{blue}{$\beta_{i+1}$}};
			\node at (1.8, .4){$\beta_{i, L}$};
			\node at (2.7, 2.1){$\beta_{i, R}$};
			\end{tikzpicture}
			\qquad
			\begin{tikzpicture}
			\node[anchor=south west,inner sep=0] at (0,0) {\includegraphics[scale=0.35]{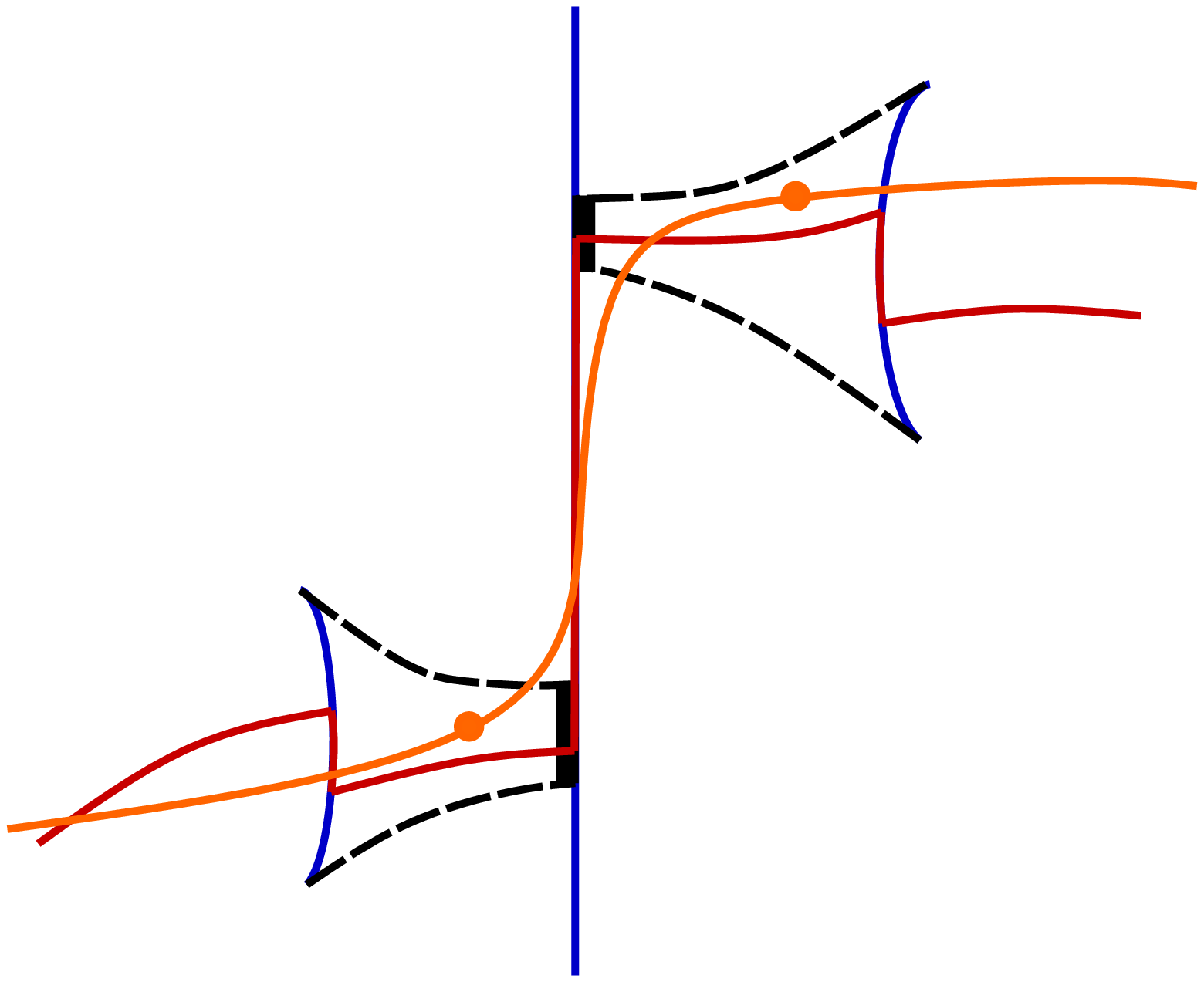}};
			\node at (2.3, 2.15){\textcolor{blue}{$\eta_i$}};
			\node at (1.3, 2.15){\textcolor{blue}{$\eta_{i-1}$}};
			\node at (4.3, 2.15){\textcolor{blue}{$\eta_{i+1}$}};
			\node at (2.2, 3.4){$\eta_{i, R}$};
			\node at (3.1, 1.1){$\eta_{i, L}$};
			\node at (3.5, 4){\textcolor{orange}{$w_{i+1}$}};
			\node at (2.1, 1.6){\textcolor{orange}{$w_{i}$}};
			\node at (5, 2.8){\textcolor{red}{$\tau'$}};
			\node at (5, 4){\textcolor{orange}{$\sigma$}};
			\end{tikzpicture}
		\end{center}
		\caption{Before and after shearing}
		\label{pic4}
	\end{figure}
	
	Denote the sequence of points $\{w_{i}\}_{i \in \mathbb{N}}$ on $\sigma$ such that each $w_{i}$ is the first point on $\sigma$ entering the $L$-neighborhood of $\eta_i$ from left. Since $\pi_{\eta_i}(q'_L) \in \eta_{i,L}$, by Corollary \ref{Crossing Fellow Travel}, for each $i$ we have
	\begin{align*}
	d(\pi_{\eta_i}(w_i), \eta_{i,L}) \le d (\pi_{\eta_i}(w_i), \pi_{\eta_i}(q'_L)) \le 2L.
	\end{align*}
	Moreover, all these points are equivalent under translation of $\sigma$, i.e, for any $i$ we have
	\begin{align*}
	d(w_{i}, w_{i+m}) = \ell_{\X}(T_\alpha^{-1}(\tau)) = \ell_{T_{\alpha}\X}(\tau).
	\end{align*}  
	By Lemma \ref{1-Lip}, we know projection $\pi_{\eta_i}$ is 1-Lipschitz for any $i$. Since $\pi_{\eta_i}(\eta_j) \subset \eta_{i,R}$, for any $i < j$ we have
	\begin{align}
	d(w_{i}, w_{j}) \ge d(\eta_{i,L}, \eta_{i,R}) - 4L \ge \lfloor D_i - 4L \rfloor_0.
	\end{align}
	See Figure \ref{pic4} for an illustration.
	
	We use $ a \ll b$ to denote that $\sigma$ goes through the point $a$ first and then the point $b$ from left to right.	We claim for any $i$ such that $D_i > 4L$ and for any $j > i$, we have $w_{i} \ll w_{j}$. Indeed, suppose $w_{j} \ll w_{i}$, then by definition the geodesic segment from $w_j$ to $w_i$ completely lies outside the $L$-neighborhood of $\eta_i$, and this means $d(\pi_{\eta_i}(w_{i}), \pi_{\eta_i}(w_{j})) \le L$ because geodesics are strongly contracting in $\mathbb H^2$, see section \ref{upper half plane properties}. Meanwhile, we know $d(\eta_{i,R}, \pi_{\eta_i}(w_{j})) \le  d(\pi_{\eta_i}(\eta_j), \pi_{\eta_i}(w_{j})) \le d(\eta_j, w_{j}) = L$ because $\pi_{\eta_i}(\eta_j) \subset \eta_{i,R}$ and because the projection map $\pi_{\eta_i}$ is $1$-Lipschitz. Combining with the previous fact $d(\pi_{\eta_i}(w_i), \eta_{i,L}) \le 2L$, we conclude
	\begin{align*}
	    D_i &\le d(\eta_{i,L}, \eta_{i,R}) \\
	    &\le d(\pi_{\eta_i}(w_i), \eta_{i,L}) + d(\pi_{\eta_i}(w_{i}), \pi_{\eta_i}(w_{j})) + d(\eta_{i,R}, \pi_{\eta_i}(w_{j})) \\
	    &\le 2L+L+L = 4L.
	\end{align*}
	And this contradicts $D_i > 4L$. Therefore, we have a pattern of ordering
	\begin{align*}
	\cdots \ll w_{0} \ll w_{1} \ll \cdots \ll w_{m-1} \ll w_{m} \ll \cdots
	\end{align*}
	on $\sigma$ provided that each $D_i > 4L$. In this case we have
	\begin{align} \label{almost hyperbolic length lower bound}
	\ell_{T_{\alpha}\X}(\tau) = d(w_{0}, w_{m}) \ge \sum_{i=1}^{m} d(w_{i-1}, w_{i}) \ge \sum_{i=1}^{m} \lfloor D_i - 4L \rfloor_0.
	\end{align}
	If for some $i$, $D_i \le 4L$, we can delete the point $w_{i}$ from our sequence and we only need to measure $d (w_{i-1}, w_{i+1})$ instead. The same result (\ref{almost hyperbolic length lower bound}) holds. Replacing $4L$ by $L$ gives us the lower bound in (\ref{estimate hyerbolic length}).
\end{proof}

While the above Theorem \ref{Hyperbolic Length} no longer has multiplicative error, we are not yet able to provide an effective lower bound for multicurves with mixed sign and with each coefficient having absolute value $\le 2$. The Proposition \ref{Hyperbolic Length again} below takes one more step and will lead to an effective lower bound for ``long'' multicurves with mixed sign and with each coefficient having absolute value $\ge 2$. Before that, we make the following two remarks that would help us establish Proposition \ref{Hyperbolic Length again}.

\begin{remark} \label{same or disjoint}
    We notice the following in the proof of Theorem \ref{Hyperbolic Length}. 
    
    Recall that we denote $i(\tau, \alpha) = m = \sum_{i=1}^k i(\tau, \alpha_i) = \sum_{i=1}^k m_i$. Let's fix $m$-many consecutive lifts in $\B$ and denote it as $\B_m \subset \B$ . Take any $\alpha_i$ in the multicurve $\alpha$ and without loss of generality, say $\beta_{i(1)}, \cdots, \beta_{i(m_i)}$ are all the lifts of this $\alpha_i$ in $\B_m$. As discussed in the proof of Theorem \ref{Hyperbolic Length}, for all $1 \le j \le m_i$, there exist isometries $\phi_{j}$ sending $\beta_{i(j)}$ to $\beta_{i(1)}$ such that all $\beta_{i(1),R}, \phi_j(\beta_{i(j), R})$ lie on $\beta_{i(1)}$. For any distinct pair $\beta_{i(j_1)},\beta_{i(j_2)}$ where $1 \le j_1 , j_2 \le m_i$, the orbits $\left\langle \psi_{\beta_{i(1)}} \right\rangle \cdot \phi_{j_1}(\beta_{i(j_1), R})$ and $\left\langle \psi_{\beta_{i(1)}} \right\rangle \cdot \phi_{j_2}(\beta_{i(j_2), R})$ are either the same or completely disjoint. That is, for any distinct pair $\beta_{i(j_1), R}, \beta_{i(j_2), R}$, either $\phi_{j_1}(\beta_{i(j_1), R}) = \phi_{j_2}(\beta_{i(j_2), R})$ or they are disjoint.
    
    Thus except the repetitive ones, we can further assume all $\beta_{i(1),R}, \psi_j(\beta_{i(j), R})$ are disjoint, lie on $\beta_{i(1)}$, and lie in a same fundamental domain of the action of $\psi_{\beta_{i(1)}}$. Denote the intersection of this fundamental domain with $\beta_{i(1)}$ as $\overline{\beta^R_{i(1)}}$, and it follows $\overline{\beta^R_{i(1)}} \subset \beta_{i(1)}$ is a path lifting of $\alpha_i$ and $\diam(\overline{\beta^R_{i(1)}}) = \ell_\X(\alpha_i)$. This means we have
    \begin{align*}
        \diam \left( \beta_{i(1),R} \cup \phi_2(\beta_{i(2), R}) \cup \cdots \cup \phi_{m_i}(\beta_{i(m_i), R}) \right) \le \ell_\X(\alpha_i).
    \end{align*}
    After removing the repetitive ones, the disjoint union of all these right neighbor projections $\phi_{j}(\beta_{j, R})$ can be arranged into $\overline{\beta^R_{i(1)}}$, a geodesic segment of diameter $\ell_\X(\alpha_i)$. One can do the same thing to all the left neighbors, and the union of all these left neighbor projections $\phi_{j}(\beta_{j, L})$ can be arranged into $\overline{\beta^L_{i(1)}}$, a geodesic segment of diameter $\ell_\X(\alpha_i)$
\end{remark}

\begin{remark}\label{all in segment}
    Continuing on Remark \ref{same or disjoint}, we recall there are $m_i$ many intersections points between $\alpha_i$ and $\tau$, and let's denote the set of these points on the surface as $X_i= \{x_1, \cdots, x_{m_i}\}$. 
    On one hand, we can lift $X_i$ to $Y_i = \{y_1, \cdots, y_{m_i}\}$, where each $y_j = \tilde{\tau} \cap \beta_{i(j)}, \beta_{i(j)} \in \B_m$.
    On the other hand, we can lift $X_i$ to $Z_i = \{z_1, \cdots, z_{m_i}\}$ where each $z_j = \phi_j(y_j)$, so all points in $Z_i$ lie in a geodesic segment of diameter $\ell_\X(\alpha_i)$, namely, $\overline{\beta^R_{i(1)}}$.

	For any of these intersection points $z_j$, we denote the corresponding lift of $\tau$ as $\tilde{\tau}_j$. The left neighbor of $z_j$ is defined to be the previous $\tilde{\alpha} \in \A$ that $\tilde{\tau}_j$ intersects, and the right neighbor of $z_j$ is the next $\tilde{\alpha} \in \A$ that $\tilde{\tau}_j$ intersects. Now, we notice the union of all these right neighbors projections of $Z_i$ are exactly the union of right neighbor projections we considered in Remark \ref{same or disjoint} and lie in $\overline{\beta^R_{i(1)}}$. Similarly for the left neighbors.
	
\end{remark}

\begin{proposition}\label{Hyperbolic Length again}
	Given a multicurve $\alpha = \sum_{i=1}^k a_i \alpha_i$ and a simple closed curve $\tau$ on any hyperbolic structure $\X$. Let $K \in (0,1)$ be a constant, we have
	\begin{align}
		\ell_{T_\alpha \X}(\tau) \ge \sum_{i=1}^k \min\{\mathcal L_1^i, \mathcal L_2^i\}  \label{hyperbolic length lower bound effective}
	\end{align}
	where
	\begin{align*}
		&\mathcal L_1^i = i(\tau, \alpha_i) \cdot \Bigl\lfloor{(|a_i|-2+K)\cdot \ell_{\X}(\alpha_i)-2\ell_\X(\tau) - L\Bigl\rfloor}_0, \\
		&\mathcal L_2^i = \Bigl\lfloor i(\tau, \alpha_i) - \frac{K \ell_\X(\alpha_{i}) + 4 \ell_\X(\tau)}{W(\tau)} \Bigl\rfloor_0 \cdot  \Bigl\lfloor{(|a_i|-1-K)\cdot \ell_{\X}(\alpha_i) - 2\ell_\X(\tau) - L\Bigl\rfloor}_0.
	\end{align*}
\end{proposition}
\begin{proof}	
	We will use the similar notations and ideas from Theorem \ref{Hyperbolic Length}. Recall we denote $i(\tau, \alpha) = m = \sum_{i=1}^k i(\tau, \alpha_i) = \sum_{i=1}^k m_i$. And recall $\phi_\A(\tilde{\tau})$ is a bi-infinite geodesic in $\Z_\A$, and its edges are labeled by
	\begin{align*}
		\B = \{\cdots, \beta_{-2}, \beta_{-1}, \beta_0, \beta_1, \beta_2, \cdots\}.
	\end{align*}
	Define the index function $s\colon \mathbb{N} \to \{1, \cdots, k\}$ so that $\alpha_{s(t)} \in \A$ is the simple closed curve such that $\beta_t$ is a lift of $\alpha_{s(t)}$. Let $K \in (0,1)$, and we consider two different scenarios.
	
    If $\diam(\beta_{t,L} \cup \beta_{t,R}) \le (2-K) \ell_\X(\alpha_{s(t)}) +2\ell_\X(\tau)$ for all $t$, following the argument from Theorem \ref{Hyperbolic Length}, we can set
    \begin{align*}
        D_t = (|a_{s(t)}|-2+K) \ell_{\X}(\alpha_{s(t)}) - 2\ell_\X(\tau),
    \end{align*}
    and we have $d(\eta_{t,L}, \eta_{t,R}) \ge D_t$ for all $t$. Following the same equation (\ref{almost hyperbolic length lower bound}) and the same argument gives us the lower bound
	\begin{align*}
		\ell_{T_\alpha \X}(\tau) \ge \sum_{i=1}^k i(\tau, \alpha_i) \cdot \Bigl\lfloor{(|a_i|-2+K)\cdot \ell_{\X}(\alpha_i)-2\ell_\X(\tau) - L\Bigl\rfloor}_0.
	\end{align*}
	
	In the second scenario, we have $\diam(\beta_{t,L} \cup \beta_{t,R}) \ge (2-K) \ell_\X(\alpha_{s(t)}) +2\ell_\X(\tau)$ for some $t$. Let's $i = s(t)$ for simplicity. For this $t$, since both $\diam(\beta_{t,L})$ and $\diam(\beta_{t,R})$ are bounded by $\ell_\X(\alpha_{i})$ respectively, we have both
	\begin{align*}
	    \diam(\beta_{t,L}), \diam(\beta_{t,R }) \ge (1-K) \ell_\X(\alpha_{i}).
	\end{align*}
	Since the $\beta_{t,L}$ is exhausting an interval length of at least $(1-K) \ell_\X(\alpha_i)$, as we discussed above in Remark \ref{same or disjoint}, the diameter of the union of all other right neighbor projections is bounded by $\ell_\X(\alpha_{i}) - (1-K) \ell_\X(\alpha_{i})$, that is, $K \ell_\X(\alpha_{s(t)})$. Similarly,
	the diameter of the union of all left neighbor projections except $\beta_{t,L}$, is bounded by $K \ell_\X(\alpha_{i})$.
	
	Denote $i = s(t)$ and let $\beta_{i(1)}, \cdots, \beta_{i(m_i)}$ denote distinct lifts of $\alpha_{i}$ in $\B_m$ with $\beta_{i(1)} = \beta_t$, see Remark \ref{same or disjoint}. Define $X_i, Y_i, Z_i$ just as in Remark \ref{all in segment}. We say a $z_j$ is in vain if its left neighbor is $\beta_{t-1}$ and its right neighbor is $\beta_{t+1}$, and we say $z_j$ is effective otherwise. Notice any points in $Z_i$ is within distance $2 \ell_\X(\tau)$ of its left neighbor projection and its right neighbor projection, see the proof of Theorem \ref{Hyperbolic Length}. 
	
	If $\beta_{t, L} \cap \beta_{t, R}$ is empty, all points in vain lie in a geodesic segment of length $4 \ell_\X(\tau)$. If $\beta_{t, L} \cap \beta_{t, R}$ is nontrivial, since
	\begin{align*}
		(2-K) \ell_\X(\alpha_{i}) +2\ell_\X(\tau) \le \diam(\beta_{t,L} \cup \beta_{t,R}) \le 2 \ell_\X(\alpha_{i}) +2\ell_\X(\tau),
	\end{align*}
	we have $\diam(\beta_{t,L} \cap \beta_{t,R}) \le K \ell_\X(\alpha_{i})$, and all points in vain is within $2 \ell_\X(\tau)$-neighborhood of $\beta_{t, L} \cap \beta_{t, R}$. In any case, all points in vain can be arranged in a geodesic segment that has length bounded by $\diam(\beta_{i,L} \cap \beta_{i,R})+4\ell_\X(\tau)$. By Collar Lemma \ref{Collar Lemma}, there are at most $\frac{K \ell_\X(\alpha_{i}) + 4 \ell_\X(\tau)}{W(\tau)}$ many intersections points in vain. 
	
	\begin{figure}[h]
		\begin{center}
			\begin{tikzpicture}
				\node[anchor=south west,inner sep=0] at (0,0) {\includegraphics[scale=0.6]{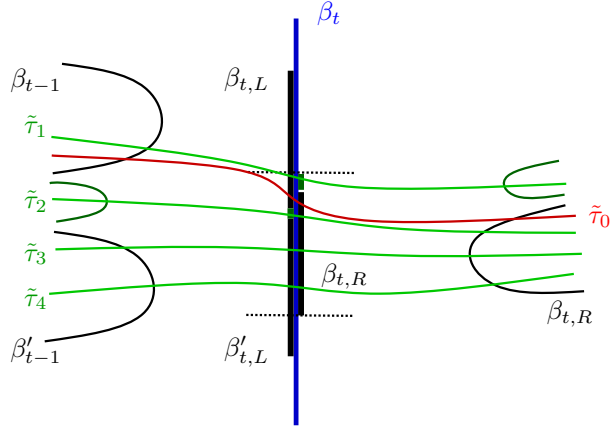}};
				\node at (3.8, 5.5){\textcolor{blue}{$\beta_t$}};
				\node at (2.7, 4.6){$\beta_{t,L}$};
				\node at (2.7, 1){$\beta_{t,L}'$};
				\node at (-0.1, 4.6){$\beta_{t-1}$};
				\node at (-0.1, 1){$\beta_{t-1}'$};
				\node at (4, 2){$\beta_{t,R}$};
				\node at (7, 1.5){$\beta_{t,R}$};
				\node at (-0.1, 4){\textcolor{OliveGreen}{$\tilde{\tau}_1$}};
				\node at (-0.1, 3){\textcolor{OliveGreen}{$\tilde{\tau}_2$}};
				\node at (-0.1, 2.3){\textcolor{OliveGreen}{$\tilde{\tau}_3$}};
				\node at (-0.1, 1.7){\textcolor{OliveGreen}{$\tilde{\tau}_4$}};
				\node at (7.4, 2.8){\textcolor{red}{$\tilde{\tau}_0$}};
			\end{tikzpicture}	
		\end{center}
		\caption{In between the two dotted lines is the geodesic segment $\overline{\beta^R_{t}}$. $\beta_{t-1}'$ and $\beta_{t-1}$ differ by $\psi_{\beta_t}$, thus their projections $\beta_{t, L}, \beta_{t, L}'$ differ by $\psi_{\beta_t}$. $\tilde{\tau}_0$ is in vain and thus is not ``counted'' in $T_{i}(K)$. Lifts like $\tilde{\tau}_1, \tilde{\tau}_2, \tilde{\tau}_3, \tilde{\tau}_4$ are effective and hence will realize a translation distance no less than $D_{i}(K)$ after twisting. $\tilde{\tau}_1$ is of case $(1)$, $\tilde{\tau}_2$, $\tilde{\tau}_3$, and $\tilde{\tau}_4$ is of case $(2)$.}
		\label{pic5}
	\end{figure}
	
	Thus, realized by $\alpha_i$ and $\beta_t$, there are at least $T_{i}(K)$ many effective intersection points, where
	\begin{align*}
		T_{i}(K) = \Bigl\lfloor i(\alpha_{i}, \tau) - \frac{K \ell_\X(\alpha_i) + 4 \ell_\X(\tau)}{W(\tau)} \Bigl\rfloor_0.
	\end{align*}
	For each effective intersection point $z_j$, it's of exactly one of the following cases.
	\begin{enumerate}
	    \item Its left neighbor is $\beta_{t-1}$ where its projection is $\beta_{t,L}$, and its right neighbor projection is being squeezed into an interval of length bounded by $K \ell_\X(\alpha_{i})$.
	    \item Its left neighbor projection is being squeezed into an interval of length bounded by $K \ell_\X(\alpha_{i})$, and its right neighbor is $\beta_{t+1}$ where its projection is $\beta_{t,L}$.
	    \item Both of its left and right neighbor projection are being squeezed into an interval of length bounded by $K \ell_\X(\alpha_{i})$ respectively.
	\end{enumerate}
	In any case, the diameter of the union of its left and right projection is bounded by 
	\begin{align*}
	(1+K)\ell_\X(\alpha_{i}) + 2\ell_\X(\tau) =	\max\left\{\ell_\X(\alpha_{i}) + K \ell_\X(\alpha_{i}) + 2\ell_\X(\tau), 2K \ell_\X(\alpha_{i}) + 2\ell_\X(\tau) \right\},
	\end{align*}
	which is like the upper bound for $\diam(\beta_{i,L} \cup \beta_{i,R})$ in the proof of Theorem \ref{Hyperbolic Length}. Apply the same argument from Theorem \ref{Hyperbolic Length} about ``$D_i$'', for any effective intersection point $z_j$, and for its corresponding $y_j$ and $\beta_{i(j)}$, it realizes a distance no less than  
	 \begin{align*}
	 	D_{i}(K) = \floor{(|a_{i}|-1-K)\cdot \ell_{\X}(\alpha_{i}) - 2\ell_\X(\tau) - L}_0
	 \end{align*}
 	after twisting.
	
 	For example, let's consider the situation in Figure \ref{pic5}. $z_0$ is the only one in vain and is not ``included'' in $T_{i}(K)$. $z_1, z_2$ realize translation distances no less than 
 	\begin{align*}
 		\floor{(|a_i|-1)\cdot \ell_{\X}(\alpha_{i}) - 2\ell_\X(\tau) - L}_0
 	\end{align*}
 	 after twisting, because the union of projections of their left and right neighbors are bounded by $\ell_{\X}(\alpha_{i})$. $z_3, z_4$ realize translation distances no less than 
 	 \begin{align*}
 	 	\floor{(|a_{i}|-1-K)\cdot \ell_{\X}(\alpha_{i}) - 2\ell_\X(\tau) - L}_0
 	 \end{align*} 
  	after twisting since the union of projections of their left and right neighbors are bounded by $(1+K)\ell_{\X}(\alpha_{i})$.
  	
 	Finally, following the same procedure from Theorem \ref{Hyperbolic Length} and only counting the sum of minimum distances realized by effective intersection points, we have
	\begin{align*}
		\ell_{T_\alpha \X}(\tau) \ge \sum_{i=1}^k T_i(K)  \cdot D_i(K).
	\end{align*}
	This gives us the desired result.
\end{proof}

Notice that the bounds in Proposition \ref{hyperbolic length not using}, Theorem \ref{Hyperbolic Length}, Proposition \ref{Hyperbolic Length again} involve both the lengths $\ell_\X(\alpha_i)$ and the intersection numbers $i(\tau, \alpha_i)$. Note also that the lower bounds we obtain are all vacuous in the case where the multicurve $\alpha$ is of mixed sign with all coefficients having absolute value 1. The following example shows that there cannot exist a general lower bound, on the order of $\sum_{i=1}^k i(\tau, \alpha_i) \ell_\X(\alpha_i)$ as in above results, that is effective in this case.

\begin{figure}[h]
	\begin{center}
		\begin{tikzpicture}
			\node[anchor=south west,inner sep=0] at (0,0) {\includegraphics[scale=0.4]{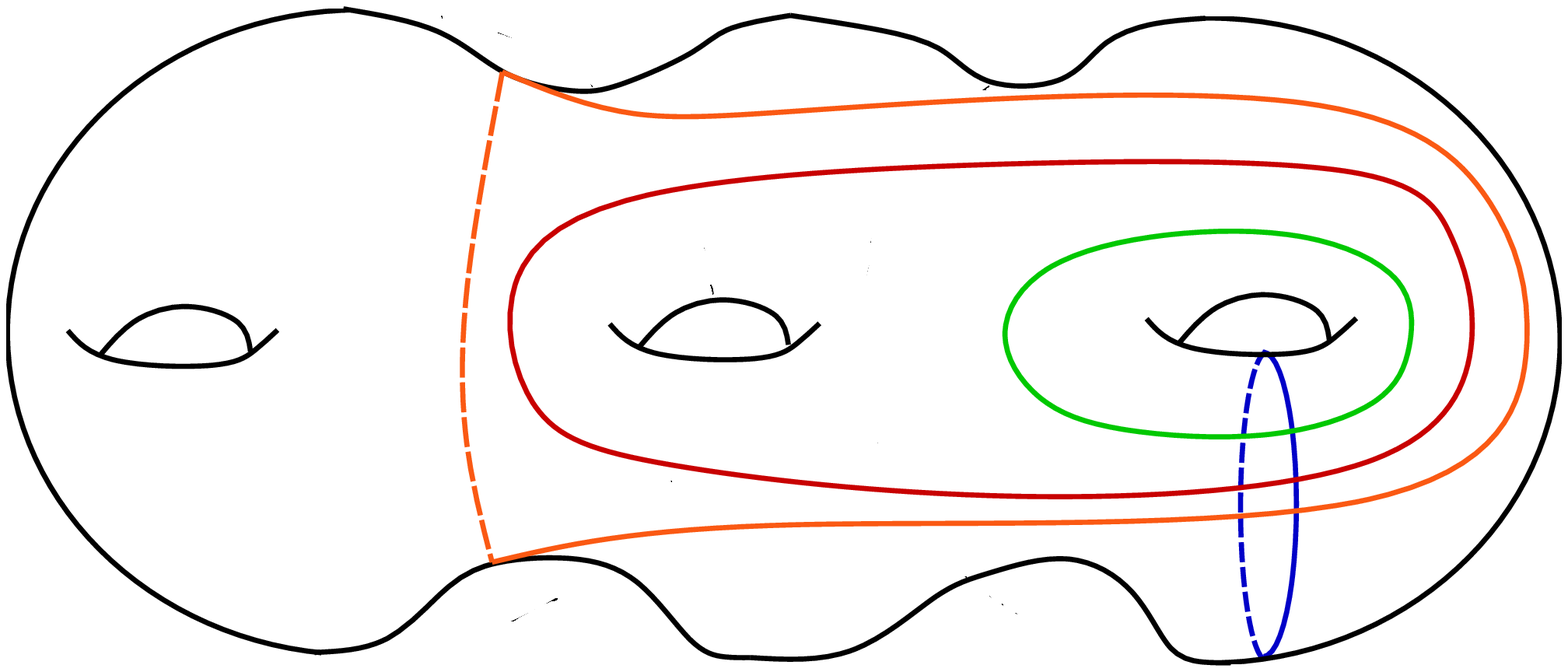}};
			\node at (2.2, 2.7){\textcolor{orange}{$\alpha$}};
			\node at (4.6, 2.3){\textcolor{red}{$\beta$}};
			\node at (5, 1.8){\textcolor{OliveGreen}{$\tau$}};
			\node at (7, 0.5){\textcolor{blue}{$\eta$}};
			
		\end{tikzpicture}	
	\end{center}
	\caption{}
	\label{pic6}
\end{figure}

\begin{example}\label{counterexample}
Consider simple closed curves $\alpha, \beta, \eta, \tau$ on a hyperbolic surface $\X$ as in the Figure \ref{pic6} above. We can define two sequences of simple closed curves
	\begin{align*}
		\alpha_i = T_\eta^i(\alpha), \beta_j = T_\eta^{j}(\beta),
	\end{align*}
and note that the lengths $\ell_\X(\alpha_n), \ell_\X(\beta_n)$ and intersection numbers $i(\tau, \alpha_n), i(\tau, \beta_n)$ all tend to infinity with $n$. Denote the multicurve $\gamma_n = \alpha_n-\beta_n$ for each $n$, then $\{\gamma_n\}_{n \in \mathbb{N}}$ is a sequence of multicurves with mixed sign and with each coefficient having absolute value equal to 1. Hence neither Proposition \ref{hyperbolic length not using}, Theorem \ref{Hyperbolic Length}, nor Proposition \ref{Hyperbolic Length again} provides an effective lower bound on $\ell_{T_{\gamma_n}\X}(\tau)$. One can use train tracks \cite{MR11410.2307/j.ctt1bd6knr4770} to study the images $T_{\gamma_n}(\tau)$ is, and then use length formula (\ref{length formula}) to verify that 
\begin{align*}
	\ell_{T_{\gamma_n}\X}(\tau) \stackrel{*}{\asymp} \ell_{\X}(\gamma_n) = \ell_{\X}(\alpha_n)+\ell_{\X}(\beta_n)
\end{align*} 
up to an uniform multiplicative error for all $n \in \mathbb{N}$. In particular, we notice the intersection numbers $i(\tau, \alpha_n), i(\tau, \beta_n)$, which go to infinity as $n$ goes to infinity, do not play any roles in the length of  $\ell_{\X}(\gamma_n)$. Thus there does not exist a constant $\lambda$ such that
\begin{align*}
    \ell_{T_{\gamma_n}\X}(\tau) \ge \lambda i(\tau, \alpha_n) \ell_{\X}(\alpha_n)+ \lambda i(\tau, \beta_n) \ell_{\X}(\beta_n).
\end{align*}
\end{example}

\section{Coarse distance} \label{Coarse distance}
In this section we will adopt the results from the previous section to establish a coarse distance formula, estimating how far a point in Teichm\"{u}ller space moves after applying a twist.

The lemma below provides a lower bound of the length of $\ell_{T_{\alpha} \X}(\kappa)$ where $\alpha$ is a multicurve of mixed sign where each coefficient has absolute value $\ge 2$, and $\kappa$ is a particular curve in the short marking $\mu_\X$.
\begin{lemma}\label{assumption on length and intersection number}
	Fix some $\epsilon > 0$, there exists $E, Q > 0$ depends on $\epsilon$ so the following holds true. Given any $\X \in \T_{g,n}^\epsilon$, let $\alpha = \sum_{i=1}^k a_i \alpha_i$ be a multicurve where each coefficient has absolute value $\ge 2$, and satisfying $\ell_{\X}(\alpha) \ge E$. Let $j$ denote an index such that $|a_j|\ell^2_\X(\alpha_{j}) = \max_{1 \le i \le k} |a_i|\ell^2_\X(\alpha_i)$. Let $\kappa \in \mu_\X$ be a marking curve such that $i(\alpha_j, \kappa) = \max_{\eta \in \mu_\X}i(\alpha_j, \eta)$, we have
	\begin{align}
		\ell_{T_\alpha \X}(\kappa) \ge \frac{1}{Q} \cdot i(\alpha_j, \kappa)|a_j|\ell_\X(\alpha_j).
	\end{align} 
\end{lemma}
\begin{proof}
	Since $k \le \frac{h}{2}, \ell_\X(\alpha) \ge E$ and since any simple closed curve has length $\ge \epsilon$, we have 
	\begin{align*}
		|a_j|\ell_\X(\alpha_j) \ge \sqrt{|a_j|^2\ell^2_\X(\alpha_j)} \ge \sqrt{|a_j|\ell^2_\X(\alpha_j)} \ge \sqrt{\frac{2 \epsilon E}{h}}.
	\end{align*}

	First, we consider the case $|a_j| \ge 3$. Since $|a_j|-2 \ge \frac{1}{3}|a_j|$ for any $|a_j| \ge 3$, by assuming $E \ge \frac{18h(2N+L)^2}{\epsilon}$, we have
	\begin{align*}
		&(|a_j|-2)\ell_\X(\alpha_j) \ge \frac{1}{3} |a_j|\ell_\X(\alpha_j) \ge \frac{1}{3} \cdot \sqrt{\frac{\epsilon E}{k}} \ge 2(2N+L), \\ &(|a_j|-2)\ell_\X(\alpha_j) -2N - L \ge \frac{1}{6} |a_j| \ell_\X(\alpha_j).
	\end{align*}
	Thus for this particular $j$ and $\kappa$, since $\ell_\X(\kappa) \le N$ we have
	\begin{align*}
		i(\alpha_j, \kappa) \left[ (|a_j|-2)\ell_\X(\alpha_j) -2\ell_\X(\kappa) -L \right] \ge \frac{1}{6} \cdot i(\alpha_j, \kappa)|a_j|\ell_\X(\alpha_j).
	\end{align*} 
	Apply Theorem \ref{Hyperbolic Length}, we conclude
	\begin{align*}
		\ell_{T_\alpha \X}(\kappa) \ge \frac{1}{6} \cdot i(\alpha_j, \kappa)|a_j|\ell_\X(\alpha_j).
	\end{align*}

	Now, we consider the case $|a_j| = 2$. Recall by fixing $\epsilon$, any short marking curve corresponding to any $\X \in \T_{g,n}^\epsilon$ has length bounded on top by $N$. By Collar Lemma \ref{Collar Lemma}, there exists a $W$ depends on $\epsilon$ such that collar width of any short marking curve corresponding to any $\X \in \T_{g,n}^\epsilon$ has length bounded below by $W$. This means we have
	\begin{align*}
	\frac{K \ell_\X(\alpha_{j}) + 4 \ell_\X(\kappa)}{W(\kappa)} \le \frac{hKCNi(\alpha_{j}, \kappa) + 4N}{W}.
	\end{align*}
	Since we have $\ell_\X(\alpha_j) \ge \frac{1}{2}\sqrt{\frac{2\epsilon E}{h}}$, by length formula (\ref{length formula}),
	\begin{align*}
	i(\alpha_j, \kappa) \ge \frac{\ell(\alpha_j)}{hCN} \ge \frac{1}{hCN}\sqrt{\frac{\epsilon E}{2h}}.
	\end{align*}
	Let $E \ge \frac{32hN^2}{\epsilon K^2}$ so that
	\begin{align*}
	    hKCN i(\alpha_j, \kappa) \ge 4N,
	\end{align*}
	and take the constant $K = \min\{\frac{W}{4hCN}, \frac{1}{2}\}$, we have
	\begin{align}
	\frac{hKCNi(\alpha_{j}, \kappa) + 4N}{W} \le \frac{2hCNK}{W} i(\alpha_{j}, \kappa) \le \frac{1}{2} i(\alpha_{j}, \kappa). \label{mess 1}
	\end{align}
	Moreover, by further assuming $E \ge \frac{12^2h(2N+L)^2}{2\epsilon}$, we have $\ell_{\X}(\alpha_j) \ge 6(2N+L)$. Since $|a_j| = 2, K \le \frac{1}{2}$, we have
	\begin{align}
		(|a_j|-1-K)\cdot \ell_{\X}(\alpha_i)-2N - L \ge \frac{1}{2}\ell_{\X}(\alpha_j)-2N-L \ge \frac{1}{3}\ell_{\X}(\alpha_j). \label{mess 2}
	\end{align}
	Now we can apply Proposition \ref{Hyperbolic Length again}. In the $\mathcal L_1$ case, $|a_j| = 2$ and the lower bound is
	\begin{align*}
		\ell_{T_{\alpha} \X}(\kappa) \ge i(\alpha_j, \kappa) \left(K\ell_{\X}(\alpha_{j}) - 2N - L\right).
	\end{align*}
	By further assuming
	\begin{align*}
	    E \ge \max\left\{\frac{12^2h(2N+L)^2}{2\epsilon}, \frac{h}{2\epsilon}\left(\frac{16hCN(2N+L)}{W}\right)^2\right\},
	\end{align*}
	we have
	\begin{align*}
	    &K\ell_{\X}(\alpha_{j}) - 2N - L \ge \frac{1}{2}K\ell_{\X}(\alpha_{j}), \\
		&\ell_{T_{\alpha} \X}(\kappa) \ge \min\left\{\frac{W}{8hCN}, \frac{1}{4}\right\} \cdot i(\alpha_j, \kappa) \ell_{\X}(\alpha_{j}).
	\end{align*}
	In the $\mathcal L_2$ case, our previous formulas (\ref{mess 1}), (\ref{mess 2}) guarantee
	\begin{align*}
		\ell_{T_{\alpha} \X}(\tau) &\ge \Big\lfloor i(\kappa, \alpha_j) - \frac{K \ell_\X(\alpha_{j}) + 2 \ell_\X(\kappa)}{W(\kappa)} \Big\rfloor_0 \cdot  \Big\lfloor(|a_j|-1-K)\cdot \ell_{\X}(\alpha_j)-2\ell_\X(\kappa) - L \Big\rfloor_0 \\
		&\ge \frac{1}{2} i(\kappa, \alpha_{j}) \cdot \frac{1}{3} \ell_{\X}(\alpha_{j}) \ge \frac{1}{6} \cdot i(\kappa, \alpha_{j}) \ell_{\X}(\alpha_{j}).
	\end{align*}
	Let
	\begin{align*}
		&E = \max\left\{\frac{12^2h(2N+L)^2}{\epsilon}, \frac{32hN^2}{\epsilon K^2}, \frac{h}{2\epsilon}\left(\frac{16hCN(2N+L)}{W}\right)^2\right\}, \\
		&Q = \min\left\{\frac{1}{6}, \frac{W}{8hCN}\right\}.
	\end{align*}
	The result follows.
\end{proof}

Now we are ready to prove the main result of this section.

\begin{theorem}[Coarse Distance Formula]\label{Coarse Distance Formula}
	Fix some $S_{g,n}$ and given any $\epsilon > 0$, there exists a constant $H > 0$ such that given any $\X \in \T_{g,n}^\epsilon$, we have
	\begin{align} \label{coarse distance formula}
		d_\T(\X, T_\alpha \X) \stackrel{+H}{\asymp} \log\left(\sum_{i=1}^{k} |a_i|\ell_\X^2(\alpha_i)\right)
	\end{align}
	for any $\alpha = \sum_{i=1}^k a_i \alpha_i \in \ML^*(\mathbb Z)$.
\end{theorem}
\begin{proof}
	By the distance formula (\ref{distance formula}) and our formula (\ref{estimate hyerbolic length}), we have
	\begin{align*}
		d_\T(\X, T_\alpha \X) 
		&\le \log \left(\frac{e^c}{\epsilon} \max_{\tau \in \mu_\X}\ell_{T_\alpha \X}(\tau)\right) \\
		&\le \log \left( \frac{Ne^c}{\epsilon} + \frac{e^c}{\epsilon} \cdot \max_{\tau \in \mu_\X} \sum_{i=1}^k i(\tau, \alpha_i)|a_i| \ell_\X(\alpha_i) \right) \\
		&\le \log \left( \frac{Ne^c}{\epsilon} + \frac{e^c}{\epsilon} \cdot \sum_{\tau \in \mu_\X}\sum_{i=1}^k i(\tau, \alpha_i)|a_i| \ell_\X(\alpha_i) \right) \\
		&= \log \left( \frac{Ne^c}{\epsilon} + \frac{e^c}{\epsilon} \cdot\sum_{i=1}^k \left(\sum_{\tau \in \mu_\X} i(\tau, \alpha_i)\right) |a_i| \ell_\X(\alpha_i) \right) \\
		&\le \log \left( \frac{Ne^c}{\epsilon} + \frac{e^c C}{\epsilon^2} \cdot\sum_{i=1}^k |a_i| \ell^2_\X(\alpha_i) \right)
	\end{align*}
	where last inequality holds by applying the length formula (\ref{length formula}). Since we always have $|a_i| \ell^2_\X(\alpha_i) \ge \epsilon^2$, by using equality $\log(a+b) = \log(1+\frac{a}{b}) + \log(b)$ we have
	\begin{align*}
		d_\T(\X, T_\alpha \X) &\le \log\left(1+ \frac{\frac{Ne^c}{\epsilon}}{\frac{e^c C}{\epsilon^2} \cdot\sum_{i=1}^k |a_i| \ell^2_\X(\alpha_i)}\right) + \log\left(\frac{e^c C}{\epsilon^2} \cdot\sum_{i=1}^k |a_i| \ell^2_\X(\alpha_i)\right) \\
		&\le \log\left(1+\frac{N}{h \epsilon C}\right) + \log\left(\frac{e^c C}{\epsilon^2}\right) + \log\left(\sum_{i=1}^k |a_i| \ell^2_\X(\alpha_i)\right).
	\end{align*}
	This gives us the upper bound in (\ref{coarse distance formula}) after setting appropriate $H$.
	
	Now we work toward the lower bound in (\ref{coarse distance formula}). Let's assume that $\ell_\X(\alpha) \ge E$, $E$ from Lemma \ref{assumption on length and intersection number}. First, we consider the case that $\alpha$ is positive or negative, i.e., all coefficients have the same sign. In this case, by applying Proposition \ref{hyperbolic length not using}, similar to the argument obtaining the upper bound, we have
	\begin{align*}
			d_\T(\X, T_\alpha \X) &\ge \log \left( \frac{1}{Ne^c} \max_{\tau \in \mu_\X} \ell_{T_{\alpha} \X}(\tau) \right)  \\
		 &\ge \log \left( \frac{1}{ANe^c} \right) + \log \left( \max_{\tau \in \mu_\X}  \sum_{i=1}^k |a_i| i(\alpha_i, \tau) \ell_\X(\alpha_i)-N  \right) \\
		 &\ge \log \left( \sum_{i=1}^k |a_i| \ell_{X}^2(\alpha_{i}) \right) - \log \left(2hACN^2e^c \right)
	\end{align*}
	
	Now, we consider the case where $\alpha$ is of mixed sign where each coefficient has absolute value $\ge 2$. Let $j$ be the index of $\alpha$ such that $|a_j|\ell_\X^2(\alpha_{j}) = \max_{1 \le i \le k} |a_i|\ell_\X^2(\alpha_i)$ and let $\kappa \in \mu_\X$ be the curve realizes $\max_{\eta \in \mu_\X} i(\alpha_j, \eta)$, then by previous Lemma \ref{assumption on length and intersection number} we have
	\begin{align*}
		\max_{\eta \in \mu_\X} \ell_{T_{\alpha} \X}(\eta) \ge \ell_{T_{\alpha} \X}(\kappa) \ge \frac{1}{Q} \cdot  i(\alpha_j, \kappa)|a_j|\ell_\X(\alpha_j).
	\end{align*}
	Apply the length formula (\ref{length formula}), we have
	\begin{align*}
		\max_{\eta \in \mu_\X} \ell_{T_{\alpha} \X}(\eta) & \ge \frac{1}{Q} i(\alpha_j, \kappa)|a_j|\ell_\X(\alpha_j) = \frac{1}{Q} \max_{\eta \in \mu_\X} i(\alpha_j, \eta)|a_j|\ell_\X(\alpha_j) \\ &\ge \frac{1}{Q}\frac{\sum_{\eta \in \mu_\X} i(\alpha_j, \eta)}{h}|a_j|\ell_\X(\alpha_j) \ge \frac{1}{hQCN}|a_j|\ell_\X^2(\alpha_j) \\ &\ge \frac{1}{hQCN} \max_{1 \le i \le k} |a_i|\ell_\X^2(\alpha_i).
	\end{align*}
	Apply the distance formula (\ref{distance formula}), we have
	\begin{align*}
		d_\T(\X, T_\alpha \X) &\ge \log\left(\frac{1}{hQCN^2e^c} \max_{1 \le i \le k} |a_i|\ell_\X^2(\alpha_i)\right) \\
		&\ge \log\left(\frac{1}{h^2QCN^2e^c} \sum_{i=1}^{k} |a_i|\ell_\X^2(\alpha_i)\right) \\
		&=\log\left(\sum_{i=1}^{k} |a_i|\ell_\X^2(\alpha_i)\right)-\log\left(h^2QCN^2e^c\right).
	\end{align*}
	And this gives us the lower bound in (\ref{coarse distance formula}) after setting appropriate $H$. Notice the above holds true for any $\alpha \in \ML^*_{\ge E}(\mathbb{Z})$. For any $\alpha \in \ML(\mathbb Z)$ that has length bounded by $E$, we have $ \log\left(\sum_{i=1}^{k} |a_i|\ell_\X^2(\alpha_i)\right) \le \log\left(\ell_\X^2(\alpha)\right) \le 2\log(E)$.
	
	Finally, we set
	\begin{align*}
		H = \max \{&\log\left(1+\frac{N}{h \epsilon C}\right) + \log\left(\frac{e^c C}{\epsilon^2}\right), \\
		&\log \left(2hACN^2e^c \right), \log\left(h^2QCN^2e^c\right), 2\log(E)\}.
	\end{align*} 
	The result follows.
\end{proof}

\begin{remark}\label{counterexample to coarse distance formula}
	Consider Example \ref{counterexample}, if $\tau$ is chosen to be a short marking curve, using similar idea in Theorem \ref{Coarse Distance Formula}, we would have
	\begin{align*}
		d_\T(\X, T_{\gamma_n} \X) \stackrel{+H}{\asymp} \log\left(\ell_{\X}(\alpha_n)+\ell_{\X}(\beta_n)\right)
	\end{align*}
	for some $H$. This implies our coarse distance formula does not hold for this sequence of multicurves $\{\gamma_n\} \subset \ML(\mathbb{Z}) \setminus \ML^*(\mathbb{Z})$.
\end{remark}

\section{Proof of theorem \ref{Main Theorem} and corollary \ref{Main Corollary}} \label{Proof of the Main Theorem}
Assume the conditions in Theorem \ref{Main Theorem}, let $D$ be one of the three sets $D(\mathcal S)$, $M(\mathcal S)$, $D(\ML^*(\gamma))$. Since any mapping class is an isometry for $\T_{g,n}$, given any $\X, \Y \in  \T_{g,n}^\epsilon$, we notice
\begin{align}\label{balls in balls}
|D \cdot \X \cap B_{R-d_\T(\X, \Y)}(\X)| \le |D \cdot \Y \cap B_R(\X)| \le |D \cdot \X \cap B_{R+d_\T(\X,\Y)}(\X)|.
\end{align}
Also, recall that
\begin{align*}
D \cdot \X \cap B_R(\X) = \{ g \cdot \X \in \T_{g,n} \mid g \in D, d_\T(g \cdot \X, \X) \le R\}.
\end{align*}
Since twists never stabilize any point in $\T_{g,n}$, we have
\begin{align*}
|D \cdot \X \cap B_R(\X)| = |\{g \in D \mid d_\T(g \cdot \X, \X) \le R\}|.
\end{align*}
For simplicity of notation, we use $t = \frac{h}{2}$ to denote half the dimension of $\T_{g,n}$. Recall we say $f(R) \stackrel{*A}\sim g(R)$ if for any $\lambda > 1$, there exists a $M(\lambda)$ such that $\frac{1}{\lambda A} \le \frac{g(R)}{f(R)} \le \lambda A$ for any $R \ge M(\lambda)$. We are now ready to prove Theorem \ref{Main Theorem} and Corollary \ref{Main Corollary}. For each case and for any $\lambda >1$, we will compute the corresponding $M(\lambda)$.
\begin{proof}[Proof of Theorem \ref{Main Theorem}]
Let $\gamma = \sum_{i=1}^k a_i \gamma_i$ be a multicurve, then $c_\alpha = c_\gamma$ for any $\alpha \in \ML^*(\gamma)$. We consider the corresponding set of twists around curves of topological type $\gamma$
\begin{align*}
D(\ML^*(\gamma)) = \{T_\alpha \mid \alpha \in \ML^*(\gamma) \}.
\end{align*}
Define
\begin{align*}
& \mathcal S_R^\pm = \left\{\alpha = \sum_{i=1}^k a_i \alpha_i \in \ML^*(\gamma) \mid \sum_{i=1}^k |a_i| \ell^2_{\X}(\alpha_i) \le e^{R \pm H}\right\}, \\
& \mathcal S_R^{++} = \left\{\alpha = \sum_{i=1}^k a_i \alpha_i \in \ML^*(\gamma) \mid \ell_\X(\alpha) \le \sqrt{c_\alpha} \cdot e^{(R+H)/2}\right\}, \\
& \mathcal S_R^{--} = \left\{\alpha = \sum_{i=1}^k a_i \alpha_i \in \ML^*(\gamma) \mid \ell_\X(\alpha) \le e^{(R-H)/2}\right\}.
\end{align*}
By the coarse distance formula (\ref{coarse distance formula}), we have 
\begin{align*}
|\mathcal S_R^-| \le |D(\ML^*(\gamma)) \cdot \X \cap B_R(\X)| \le |\mathcal S_R^+|.
\end{align*}
Since
\begin{align*}
\sum_{i=1}^k |a_i| \ell^2_{\X}(\alpha_i) \le \sum_{i=1}^k |a_i|^2 \ell^2_{\X}(\alpha_i) \le (\ell_\X(\alpha))^2 
\end{align*}
we have $\mathcal S_R^{--} \subset \mathcal S_R^-$. Moreover, by Schwartz inequality, we have
\begin{align}\label{bounding relation 1} 
(\ell_\X(\alpha))^2 = (\sum_{i=1}^k |a_i| \ell_\X(\alpha_i))^2 \le  (\sum_{i=1}^k |a_i|) \cdot \sum_{i=1}^k |a_i| \ell^2_\X(\alpha_i) = c_\alpha \cdot \sum_{i=1}^k |a_i| \ell^2_\X(\alpha_i)
\end{align}
so $\mathcal S_R^{+} \subset \mathcal S_R^{++}$. Together this means
\begin{align*}
|\mathcal S_R^{--}| \le |D(\ML^*(\gamma)) \cdot \X \cap B_R(\X)| \le |\mathcal S_R^{++}|.
\end{align*}
Mirzakhani's counting formula (\ref{repharse counting formula}) tells us for any $\lambda > 1$, we have
\begin{align*}
&|\mathcal S_R^{--}| = s_X(e^{(R-H)/2}, \gamma) \ge \frac{1}{\lambda} \cdot n_X( \gamma) \cdot e^{t(R-H)}, \\
&|\mathcal S_R^{++}| = s_X(\sqrt{c_\alpha} \cdot e^{(R+H)/2}, \gamma) \le \lambda \cdot n_X( \gamma) \cdot c_\gamma^t \cdot e^{t(R+H)}
\end{align*}
whenever $R \ge M(\lambda) = 2\log(r_X(\gamma, \lambda))+H$. This means
\begin{align*}
\frac{1}{\lambda} \cdot n_X(\gamma) \cdot e^{t(R-H)} \le |D(\ML^*(\gamma)) \cdot \X \cap B_R(\X)| \le \lambda \cdot c_\gamma^t \cdot n_X(\gamma) \cdot e^{t(R+H)}  
\end{align*}
whenever $R \ge M(\lambda)$. By equation (\ref{balls in balls}), we have
\begin{align*}
\frac{1}{\lambda \cdot e^{t(d_\T(\X,\Y)+H)}} \le \frac{|D(\ML^*(\gamma)) \cdot \Y \cap B_R(\X)|}{n_X(\gamma)e^{t R}} \le \lambda \cdot c_\gamma^t  \cdot e^{t(d_\T(\X,\Y)+H)}
\end{align*}
whenever $R \ge M(\lambda)$. Recall that we denote $F_\gamma(\X,\Y) =(c_\gamma)^t e^{t d_\T(\X,\Y)}$. By setting $J = e^{tH}$, we are done with the case $D=D(\ML^*(\gamma))$ and Theorem \ref{Main Theorem} follows.
\end{proof}

\begin{proof}[Proof of Corollary \ref{Main Corollary}]
We observe in the above proof of Theorem \ref{Main Theorem}, when $\gamma$ is a simple closed curve, $c_\gamma =1$ and for $R \ge 2\log(r_X(\gamma, \lambda))+H$ we have
\begin{align*}
\frac{1}{\lambda \cdot e^{t(d_\T(\X,\Y)+H)}} \le \frac{|D(\mathcal S(\gamma)) \cdot \Y \cap B_R(\X)|}{n_X(\gamma)e^{t R}} \le \lambda \cdot e^{t(d_\T(\X,\Y)+H)}.
\end{align*}
Summing up all the topological types of simple closed curves, we have
\begin{align*}
\frac{1}{\lambda \cdot e^{t(d_\T(\X,\Y)+H)}} \le \frac{|D(\mathcal S) \cdot \Y \cap B_R(\X)|}{n_X(\mathcal S) e^{t R}} \le \lambda \cdot e^{t(d_\T(\X,\Y)+H)}
\end{align*}
whenever $R \ge M(\lambda) = 2\log(r_X(\mathcal S, \lambda))+H$. Recall we denote $F(\X,\Y) = e^{t d_\T(\X,\Y)}$ and $J = e^{tH}$. Thus we are done with the case $D=D(\mathcal S)$ and the first result of Corollary \ref{Main Corollary} follows.

Now, we consider the set of Dehn twists with powers
\begin{align*}
M(\mathcal S) = \{T_\alpha^n \mid \alpha \in \mathcal S, n \in \mathbb{Z}\setminus\{0\} \}.
\end{align*}
Define
\begin{align*}
& M_R^\pm = \left\{T^n_{\alpha} \in M(\mathcal S) \mid |n|\ell^2_{\X}(\alpha) \le e^{R \pm H} \right\},  \\
& \mathcal S^\pm_{R,n} = \left\{\alpha \in \mathcal S \mid \ell_{\X}(\alpha) \le \frac{e^{(R \pm H)/2}}{\sqrt{|n|}}\right\} 
\end{align*}
so that
\begin{align*}
	|M_R^\pm| = \sum_{n \in \mathbb{Z} \setminus \{0\}} |S_{R,n}^\pm| = 2 \cdot \sum_{n \in \mathbb{N}} |S_{R,n}^\pm|.
\end{align*}
Thus we only need to consider $n \in \mathbb{N}$. Since we are in $\T_{g,n}^\epsilon$, $S_{R,n}^\pm$ is empty when $n \ge \frac{e^{R \pm H}}{\epsilon^2}$, we have
\begin{align*}
		|M_R^\pm| =  2 \cdot \sum_{n=1}^{\frac{e^{R \pm H}}{\epsilon^2}} |S_{R,n}^\pm|.
\end{align*}

Fix some $\lambda > 1$ and let's assume $r_X(\mathcal S, \lambda) \ge \epsilon$. Let's also asssume that $R \ge 2\log(r_X(\mathcal S, \lambda))$ so that $e^{R+H} \ge r_X^2(\mathcal S, \lambda)$. For simplicity let's denote
\begin{align*}
	a = \frac{e^{R + H}}{\epsilon^2}, b = \frac{e^{R + H}}{r_X^2(\mathcal S, \lambda)}.
\end{align*}
Then the above assumptions say $a \ge b \ge 1$.
By Corollary \ref{nr} and the Mirzakhani's counting formula (\ref{repharse counting formula}), we have
\begin{align*}
	\mathcal S^+_{R,n} = s_X(\frac{e^{(R + H)/2}}{\sqrt{n}}, \mathcal S) \le \lambda \cdot n_X(\mathcal{S}) \cdot e^{t(R+H)} \cdot \frac{1}{n^t}
\end{align*}
provided that $n \le \min \{a,b\}=b$. Notice now we have
\begin{align*}
	|M_R^+| &\le 2 \cdot \sum_{n=1}^{a} |\mathcal S^+_{R,n}| \le 2 \cdot \sum_{n=1}^{b} |\mathcal S^+_{R,n}| + 2 \cdot \sum_{n=b}^{a} |\mathcal S^+_{R,n}|
\end{align*}
where
\begin{align*}
	&\sum_{n=1}^{b} |\mathcal S^+_{R,n}| \le \lambda \cdot n_X(\mathcal{S}) \cdot e^{t(R+H)} \cdot \sum_{n=1}^{b} \frac{1}{n^t}, \\
	&\sum_{n=b+1}^{a} |\mathcal S^+_{R,n}| \le \sum_{n=b+1}^{a} |\mathcal S^+_{R,b}| \le \sum_{n=1}^{a} |\mathcal S^+_{R,b}| \le \frac{s_X(r_X(\mathcal S, \lambda), \mathcal S)}{\epsilon^2} \cdot e^{R+H}
\end{align*}
whenever $R \ge c_{\lambda} = 2\log(r_X(\mathcal S, \lambda))$. 

When $t > 1$, $\sum_{n=1}^\infty \frac{1}{n^t}$ converges and is bounded by $2$. By assuming $R$ is even larger $\left(R \ge d_\lambda = \log\left(\frac{s_X(r_X(\mathcal S, \lambda), \mathcal S)}{2 \lambda \epsilon^2 n_X(\mathcal S)}\right)\right)$, bigger exponential wins and we have
\begin{align*}
	2\lambda \cdot n_X(\mathcal{S}) \cdot e^{t(R+H)} \ge \frac{s_X(r_X(\mathcal S, \lambda), \mathcal S)}{\epsilon^2} \cdot e^{R+H}
\end{align*}
Thus we have
\begin{align*}
	|M_R^+| \le 8\lambda \cdot n_X(\mathcal{S}) \cdot e^{t(R+H)}
\end{align*}
whenever $R \ge \max\{c_\lambda, d_\lambda\}$. By the coarse distance formula (\ref{coarse distance formula}) we have
\begin{align*}
|\mathcal S_R^{--}| \le |M(\mathcal S) \cdot \X \cap B_R(\X)| \le |M_R^+|,
\end{align*}
$|\mathcal S_R^{--}|$ is from the proof of Theorem \ref{Main Theorem}. Similar to the previous cases, we have
\begin{align*}
\frac{1}{\lambda e^{t(d_\T(\X,\Y)+H)}} \le \frac{|M(\mathcal S) \cdot \Y \cap B_R(\X)|}{n_X(\mathcal S)e^{t R}} \le \lambda \cdot 8e^{t(d_\T(\X,\Y)+H)}
\end{align*}
whenever $R \ge M(\lambda) = \max\{r_X(\mathcal S, \lambda), c_\lambda, d_\lambda\}$. The second result of Corollary \ref{Main Corollary} follows.
\end{proof}
\begin{proof}[Proof of upper bound of Corollary \ref{second main corollary}]
We now consider $t=\frac{h}{2}=1$. Since in this case $\ML(\mathbb{Z})$ is 1 dimensional, we have $\ML(\mathbb{Z}) = M(\mathcal S)$.

Notice from the above proof of $M(\mathcal S)$, when $t=1$, $\sum_{n=1}^b \frac{1}{n} \le \log(b+1)$ where
\begin{align*}
	\log(b+1) =  \log(\frac{e^{R + H}}{r_X^2(\mathcal S, \lambda)}+1) \le R
\end{align*}
by assuming $r_X(\mathcal S, \lambda)$ sufficiently large. Moreover, we have
\begin{align*}
	\lambda \cdot n_X(\mathcal{S}) \cdot R \cdot e^{R+H} \ge \frac{s_X(r_X(\mathcal S, \lambda), \mathcal S)}{\epsilon^2} \cdot e^{R+H}
\end{align*}
when $R$ is large $\left(R \ge l_\lambda = \frac{s_X(r_X(\mathcal S, \lambda), \mathcal S)}{\lambda \epsilon^2 n_X(\mathcal S)}\right)$. Thus we have
\begin{align*}
	|M_R^+| \le 4\lambda \cdot n_X(\mathcal{S}) \cdot R \cdot e^{R+H}
\end{align*}
whenever $R \ge M(\lambda) = \max\{c_\lambda, l_\lambda\}$. Similar to previous arguments, this shows
\begin{align*}
	 |M(\mathcal S) \cdot \Y \cap B_R(\X)|  \stackrel{*4JF(\X,\Y)}{\preceq} n_X(\mathcal S) \cdot R \cdot e^{\frac{h}{2} R}
\end{align*}
when $\frac{h}{2}=1$.
\end{proof}

\section{Proof of theorem \ref{second main theorem} and corollary \ref{second main corollary}} \label{proof of second main theorem}
Let $\underline{\gamma}$ be a multicurve satisfying the conditions in Theorem \ref{second main theorem}. Given $s \in \mathbb{N}$, we denote
\begin{align*}
&[\underline{\gamma}, s] = \{\gamma \in [\underline{\gamma}] \mid c_\gamma = s\}, \\
&\ML([\underline{\gamma}, s]) = \bigsqcup_{\gamma \in [\underline{\gamma}, s]} \ML(\gamma),
\end{align*}
and we denote $\#[\underline{\gamma}, s]$ the number of $\gamma \in [\underline{\gamma}], c_\gamma =s$. Indeed, this number equals, up to $\Mod_{g,n}$, the number of topological types of curves composing the set $\ML([\underline{\gamma}, s])$. For any $l < s \in \mathbb{N}$, we denote
\begin{align*}
[\underline{\gamma}, s, l] = \left\{\gamma = \sum_{i=1}^k a_i \gamma_i \in [\underline{\gamma}, s] \mid |a_i| \ge l \text{\ for any i}\right\},
\end{align*}
and $\ML([\underline{\gamma}, s, l]), \#[\underline{\gamma}, s, l]$ are similarly defined as above.
\begin{lemma}
	Let $\underline{\gamma} = \sum_{i=1}^k \gamma_i$ be a multicurve with all coefficients equal to one and of maximal dimension $k= \frac{h}{2}$. For $s \ge h-2$ we have
	\begin{align}
	\#[\underline{\gamma}, s] \ge \frac{s^{k-1}}{2^{k-1}(k-1)!}.
	\end{align}
	In particular, there exists a $t$ such that for any $s \ge h-2$, we have
	\begin{align}\label{t for large s}
	\#[\underline{\gamma}, s, \frac{s}{t}] \ge \frac{1}{2}\#[\underline{\gamma}, s] \ge \frac{s^{k-1}}{2^k(k-1)!}.
	\end{align}
\end{lemma}
\begin{proof}
	The number $\#[\underline{\gamma}, s]$ equals to the number of ordered $k$-tuples $(x_1, \cdots, x_k) \in \mathbb{N}^k$ such that $\sum_{i=1}^k x_i = s$. It's a standard combinatorics fact this number is ${s-1 \choose k-1}$, which is greater than $\frac{s^{k-1}}{2^{k-1}(k-1)!}$ whenever $s \ge 2(k-1) = h-2$.
	
	For any $x=(x_1, \cdots, x_k) \in \mathbb{R}_+^k$ we denote $\delta_x$ the corresponding dirac measure and denote $\|x\| = \sum_{i=1}^k x_i$.  We define the following sets
	\begin{align*}
	& C = \{x \in \mathbb{R}_+^k \mid \|x\| = 1\}, \\
	& C^t = \{x \in C \mid x_i \ge \frac{1}{t} \text{\ for any } i \}, \\
	& C_s = \{x \in \mathbb{N}^k \mid \|x\| = s \}, \\
	& C^{\frac{s}{t}}_s = \{x \in C_s \mid x_i \ge \frac{s}{t} \text{\ for any i}\}
	\end{align*}
	where $t, s \in \mathbb{N}$. Define the measures $\delta_s = \sum_{x \in C_s} \delta_{\frac{x}{s}}$ and $\delta^{ \frac{s}{t}}_s  = \sum_{x \in C^{\frac{s}{t}}_s} \delta_{\frac{x}{s}}$. Denote the standard probability measure on $C$ as $\mu$, a classic measure theory result says the following ratio converges, and we have
	\begin{align*}
	\lim_{s \to \infty} \frac{\#[\underline{\gamma}, s, \frac{t}{s}]}{\#[\underline{\gamma}, s]} = \lim_{s \to \infty}  \frac{\delta^{\frac{s}{t}}_s(C)}{\delta_{s}(C)} = \frac{\mu(C^t)}{\mu(C)}.
	\end{align*}
	Thus by picking a $t$ large enough the second equation (\ref{t for large s}) above holds true.
\end{proof}

\begin{corollary} \label{gamma and underline gamma}
	Let $\underline{\gamma} = \sum_{i=1}^k \gamma_i$ be a multicurve with all coefficients equal to one and of maximal dimension $k= \frac{h}{2}$. For any $\gamma \in [\underline{\gamma}, s, \frac{s}{t}]$, we have
	\begin{align}
	& (\ell_\X(\gamma))^2 \ge \frac{s}{t} \cdot \sum_{i=1}^k |a_i| \ell_\X^2(\gamma_i), \nonumber \\
	& \frac{s}{t} \cdot \ell_\X (\underline{\gamma}) \le \ell_\X(\gamma) \le s \cdot \ell_\X (\underline{\gamma}). \label{bounding relation 2}
	\end{align}
	This means for any $\gamma \in [\underline{\gamma}, s, \frac{s}{t}]$ and any $\lambda > 1$, we have
	\begin{align} \label{sXLalpha}
	s_X(L, \gamma) \ge s_X(L, s \cdot \underline{\gamma}) \ge \frac{1}{\lambda} \cdot \frac{L^h}{s^h} n_X( \underline{\gamma})
	\end{align}
	for $L \ge s \cdot r_X( \underline{\gamma}, \lambda)$.
\end{corollary}
\begin{proof}
	The first two equations follow from the definition of $[\underline{\gamma}, s, \frac{s}{t}]$. The third equation follows from Corollary \ref{nr}.
\end{proof}
Now we are ready to prove the Theorem \ref{second main theorem}.
\begin{proof}[Proof of Theorem \ref{second main theorem}]
	Define
	\begin{align*}
	& \mathcal S_R = D\left(\ML([\underline{\gamma}])\right) \cdot \X \cap B_R(\X), \\
	& \mathcal S_R^- = \left\{\alpha = \sum_{i=1}^{k} a_i \alpha_i \in \ML([\underline{\gamma}]) \mid \sum_{i=1}^k |a_i| \ell^2_{\X}(\alpha_i) \le e^{R-H}\right\}, \\
	& \mathcal S_R^-(s) = \left\{\alpha = \sum_{i=1}^{k} a_i \alpha_i \in \ML([\underline{\gamma}, s]) \mid \sum_{i=1}^k |a_i| \ell^2_{\X}(\alpha_i) \le e^{R-H}\right\}, \\
	& \mathcal S_R^-(s, \frac{s}{t}) = \left\{\alpha = \sum_{i=1}^{k} a_i \alpha_i \in \ML([\underline{\gamma}, s, \frac{s}{t}]) \mid \ell_{\X}(\alpha) \le \sqrt{\frac{s}{t}} \cdot e^{\frac{R-H}{2}} \right\}.
	\end{align*}
	Notice for any $\alpha \in \mathcal S_R^-(s, \frac{s}{t})$,  by previous Corollary \ref{gamma and underline gamma} we have
	\begin{align*}
	\frac{s}{t} \cdot \sum_{i=1}^k |a_i| \ell_\X^2(\alpha_i) \le (\ell_\X(\alpha))^2 \le \frac{s}{t}\cdot e^{R-H}
	\end{align*}
	so that $S_R^-(s, \frac{s}{t}) \subset \mathcal S_R^-(s)$. Fix some $\lambda > 1$, by Theorem \ref{Counting Formula} and formulas (\ref{t for large s}), (\ref{sXLalpha}), we have
	\begin{align*}
	\left|\mathcal S_R^-(s, \frac{s}{t})\right| & \ge \sum_{\gamma \in [\underline{\gamma}, s, \frac{s}{t}]} s_X \left( \sqrt{\frac{s}{t}} \cdot e^{\frac{R-H}{2}}, \gamma \right) \\
	& \ge \frac{1}{\lambda} \cdot \#[\underline{\gamma}, s, \frac{s}{t}] \cdot \frac{\left(\sqrt{\frac{s}{t}} \cdot e^{\frac{R-H}{2}}\right)^{h}}{s^h} \cdot n_X( \underline{\gamma}) \\
	& \ge \frac{1}{\lambda} \cdot \frac{s^{k-1}}{2^k(k-1)!} \cdot \frac{e^{k(R-H)}}{s^kt^k} \cdot n_X( \underline{\gamma}) \\
	& = \frac{1}{\lambda} \cdot \frac{ n_X(\underline{\gamma}) \cdot e^{k(R-H)}}{2^k(k-1)!t^k} \cdot \frac{1}{s}
	\end{align*}
	provided that
	\begin{align*}
	 \sqrt{\frac{s}{t}} \cdot e^{\frac{R-H}{2}} \ge s \cdot r_X( \underline{\gamma}, \lambda) \text{\ and\ } s \ge h-2.
	\end{align*}
	That is,
	\begin{align*}
	h-2 \le s \le \frac{e^{R-H}}{t \cdot r^2_X(\underline{\gamma}, \lambda)}.
	\end{align*}
	Thus, we have
	\begin{align*}
	|\mathcal S_R| &\ge |\mathcal S_R^-| = \sum_{s \in \mathbb{N}} |\mathcal S_R^-(s)| \\ &\ge \sum_{s=h-2}^{ \frac{e^{R-H}}{t \cdot r^2_X(\underline{\gamma}, \lambda)}} \left|\mathcal S_R^-(s, \frac{s}{t})\right| \\
	&\ge \frac{1}{\lambda} \cdot \frac{ n_X(\underline{\gamma}) \cdot e^{k(R-H)}}{2^k(k-1)!t^k} \cdot \sum_{s=h-2}^{ \frac{e^{R-H}}{t \cdot r^2_X(\underline{\gamma}, \lambda)}} \frac{1}{s} 
	\end{align*}
By assuming $R$ is large $\left(R \ge M(\lambda) = 2\left(H+\log(tr_X(\underline{\gamma}, \lambda))+\log(h-2)\right) \right)$, the summation $\sum_{s=h-2}^{ \frac{e^{R-H}}{t \cdot r^2_X(\underline{\gamma}, \lambda)}} \frac{1}{s} \ge \frac{R}{2}$.  We now have
	\begin{align*}
	|\mathcal S_R| \ge \frac{1}{\lambda} \cdot \frac{n_X(\underline{\gamma})}{2^{k+1}(k-1)!t^k} \cdot R \cdot e^{\frac{hR}{2}}
	\end{align*}
Similar to the proof of Theorem \ref{Main Theorem}, we have
	\begin{align*}
	\left|D\left(\ML([\underline{\gamma}])\right) \cdot \Y \cap B_R(\X)\right| &\ge \left|S_{R-d_\T(\X,\Y)}\right| \\
	&\ge \frac{1}{\lambda JF(\X,\Y)} \cdot f(\underline{\gamma}) \cdot R \cdot  e^{\frac{hR}{2}}
	\end{align*}
whenever $R \ge M(\lambda)$, and 
	\begin{align*}
		f(\underline{\gamma}) = \frac{n_X(\underline{\gamma})}{2^{k+1}(k-1)!t^k}
	\end{align*}
	This concludes the proof of Theorem \ref{second main theorem}.
\end{proof}
Finally, we discuss about how  Corollary \ref{second main corollary} follows from previous results.
\begin{proof}[Proof of Corollary \ref{second main corollary}]
When $\frac{h}{2}=1$, $\ML(\mathbb{Z})$ is one dimensional. Take any simple closed curve $\underline{\gamma}$, then it's maximal dimension and $\ML(\mathbb{Z}) = \ML([\underline{\gamma}])$. As a special case of Theorem \ref{second main theorem}, we have $f(\underline{\gamma}) = n_X(\underline{\gamma})= n_X(\mathcal S)$ and
\begin{align*}
	n_X(\mathcal S) \cdot R \cdot e^{R} \stackrel{*JF(\X,\Y)}\preceq \left|D\left(\ML(\mathbb{Z})\right) \cdot \Y \cap B_R(\X)\right|
\end{align*}
This gives us the lower bound. The upper bound follows from an alternation of proof of Corollary \ref{Main Corollary}, see section \ref{Proof of the Main Theorem}. This concludes the result.
\end{proof}

\begin{remark}
	We briefly discuss about the difficultly using Theorem \ref{Coarse Distance Formula} to obtain an upper bound estimate for the coarse asymptotic rate of $|D(\ML(\mathbb Z)) \cdot \Y \cap B_R(\X)|$. When we are dealing with a conjugacy class of multicurves, say the conjugacy class of $\gamma = \sum_{i=1}^k a_i \gamma_i$, we have a ``bounding relation" (\ref{bounding relation 1}) between $\ell^2_\X(\gamma)$ and $\sum_{i=1}^k |a_i| \ell^2_\X(\gamma_i)$ depends only on the coefficients of $\gamma$, and later we use this relation to estimate the number of corresponding lattice points inside a ball of radius $R$. In the case of $\ML([\underline{\gamma}])$, $\underline{\gamma}$ being a maximal dimensional multicurve with all coefficients equal to one, we consider a subset of $\ML([\underline{\gamma}])$ with ``balanced weights" so that a uniform ``bounding relation" (\ref{bounding relation 2}) holds. This idea in fact works for multicurves with``balanced weights". Namely, for any $m \ge 0$, we can define the following subset of multicurves
	\begin{align*}
		\ML(\mathbb{Z}, m) = \{\alpha = \sum_{i=1}^k a_i \alpha_i \in \ML(\mathbb{Z}) \mid |a_i| \ge \frac{c_\alpha}{m} \text{\ for each\ } i\}
	\end{align*}
By using similar ideas one can show
\begin{align*}
	\left|D\left(\ML(\mathbb{Z}, m)\right) \cdot \Y \cap B_R(\X)\right| \stackrel{*JF(\X,\Y)}\preceq s^{\frac{h}{2}} \cdot n_X(\mathcal S) \cdot R \cdot e^{\frac{h}{2}R}
\end{align*}
However, for a sequence of multicurves $\{\gamma_j\}_{j \in \mathbb{N}}$ such that $\gamma_j$ is outside $\ML(\mathbb{Z}, j)$, the possible ``bounding relations" get more and more coarse, and does not yield a uniform upper bound as above for $\ML(\mathbb Z)$.
\end{remark}

\bibliographystyle{amsplain}
\bibliography{references}

\end{document}